\documentclass[reqno,11pt,a4paper]{amsart}
\usepackage{amssymb, array, amsmath, amscd, pdfpages, enumerate, amsthm, fixltx2e, setspace,mathtools}
\usepackage[colorlinks=true]{hyperref}
\usepackage{tcolorbox}
\usepackage[utf8]{inputenc}
\usepackage{amssymb}
\usepackage{bbm}
\usepackage{overpic}
\usepackage{xcolor}
\usepackage{graphicx} 
\usepackage{ifthen}
\usepackage{multirow}
\usepackage{cite}
\usepackage{subcaption}
\usepackage{epstopdf}
\usepackage[papersize={17cm,26cm},total={12.7cm,22.5cm},top=2cm,left=2.2cm]{geometry} 
\newcolumntype{P}[1]{>{\centering\arraybackslash}p{#1}}

\usepackage{ifthen}
\DeclareMathOperator{\re}{Re}			
\DeclareMathOperator{\im}{Im}			
\DeclareMathOperator{\diag}{diag}		

\newcommand*{\ldelim}[2]{\csname#1l\endcsname#2}   
\newcommand*{\rdelim}[2]{\csname#1r\endcsname#2}   
\newcommand*{\mdelim}[2]{\csname#1m\endcsname#2}   

\newcommand*{\C}{{\mathbb{C}}}     

\newcommand*{\R}{{\mathbb{R}}}     
\newcommand*{\N}{{\mathbb{N}}}     

\newcommand*{\me}{{\mathrm{e}}}
\newcommand*{\mi}{{\mathrm{i}}}

\newcommand{\cA}{{\mathcal A}}

\newcommand{\cF}{{\mathcal F}}

\newcommand*{\Abs}[2][default]{\ifthenelse{\equal{#1}{default}}{\left\lvert#2\right\rvert}{\ldelim{#1}{\lvert}#2\rdelim{#1}{\rvert}}}
\newcommand*{\Norm}[2][default]{\ifthenelse{\equal{#1}{default}}{\left\lVert#2\right\rVert}{\ldelim{#1}{\lVert}#2\rdelim{#1}{\rVert}}}

\newcommand*{\Iprod}[3][default]{\ifthenelse{\equal{#1}{default}}{\left\langle#2,#3\right\rangle}{\ldelim{#1}{\langle}#2,#3\rdelim{#1}{\rangle}}}
\newcommand*{\Dualpair}[3][default]{\ifthenelse{\equal{#1}{default}}{\left\langle#2,#3\right\rangle}{\ldelim{#1}{\langle}#2,#3\rdelim{#1}{\rangle}}}

\newcommand*{\Set}[2][default]{\ifthenelse{\equal{#1}{default}}{\left\{#2\right\}}{\ldelim{#1}{\{}#2\rdelim{#1}{\}}}}

\newcommand*{\conj}[1]{\overline{#1}}

\newcommand{\bmat}[1]{\begin{bmatrix}#1\end{bmatrix}}

\newcommand{\pmat}[1]{\begin{pmatrix}#1\end{pmatrix}}

\newcommand{\hiddenproof}[2]{
\ifthenelse{#1}{Hidden}{#2}
}

\newcommand{\longcomm}[1]{}

\usepackage{datetime}
\newdateformat{findate}{\THEDAY.\THEMONTH.\THEYEAR}

\newcommand{\pdt}{\partial_t}
\newcommand{\pdtt}{\partial_{tt}^2}
\newcommand{\pdx}{\partial_x}
\newcommand{\pdxx}{\partial_{xx}^2}

\newcommand{\pdxxxx}{\partial_{xxxx}^4}
\newcommand{\pdxxt}{\partial_{xxt}^3}

\newtheorem{theorem}{Theorem}[section]

\newtheorem{lem}[theorem]{Lemma}
\newtheorem{cor}[theorem]{Corollary}

\theoremstyle{definition}

\newtheorem{rem}[theorem]{Remark}
\newtheorem{ex}[theorem]{Example}

\begin{document}

\title[Exponential integrators for second-order in time PDEs]{Exponential integrators for second-order in time partial differential equations}

\author{Alexander Ostermann}
\email{alexander.ostermann@uibk.ac.at}
\address{Institut f\"ur Mathematik, Leopold-Franzens-Universit\"at Innsbruck, Technikerstra\ss e 13, A-6020 Innsbruck, Austria.}

\author{Duy Phan}
\email{duy.phan-duc@uibk.ac.at}
\address{Institut f\"ur Mathematik, Leopold-Franzens-Universit\"at Innsbruck, Technikerstra\ss e 13, A-6020 Innsbruck, Austria.}

\begin{abstract}
Two types of second-order in time partial differential equations (PDEs), namely semilinear wave equations and semilinear beam equations are considered. To solve these equations with exponential integrators, we present an approach to compute efficiently the action of the matrix exponential as well as those of related matrix functions. Various numerical simulations are presented that illustrate this approach. 
\end{abstract}

\date{\today}


\maketitle
%

\section{Introduction}
We consider semilinear damped wave equations with damping term, structural (visco-elastic) damping term, and mass term 
\begin{subequations} \label{eq-wave-gen}
\begin{align}
\pdtt u &- \alpha \Delta u - \beta \Delta (\pdt u)+ \gamma \pdt u + \delta u  = g(u) + h(\pdt u), \quad (x,t) \in \Omega \times (0,T), \\
u|_{\partial \Omega} &= 0, \quad (x,t) \in \partial\Omega \times (0,+\infty), \\
u(0,x) &= p(x), \quad \pdt u(0,x) = q(x), \qquad x \in \Omega  
\end{align}
\end{subequations} 
on a bounded and open domain $\Omega \subset \R^d$ with smooth compact boundary $\partial \Omega$. The term $\Delta (\pdt u)$ is the structural (visco-elastic) damping while the term $\pdt u$ is the damping term. We assume that $\beta, \gamma$, and $\delta$ are three non-negative coefficients. Moreover, the coefficient $\alpha$ must be positive.   The initial data $p$ and $q$ are chosen from the usual energy space $(p, q) \in H_0^1(\Omega) \times L^2(\Omega) $. Concerning the nonlinear term, we recall some particular equations from the literature:
\begin{enumerate}[(i)]
\item the perturbed sine-Gordon equation (see \cite{PataZel06,GhiMar91,CarChoDlo08})
\begin{align} \label{eq-wave-sineGordon}
\pdtt u &- \alpha \Delta u - \beta \Delta (\pdt u)   + \gamma \pdt u = \sin u;
\end{align} 
\item the perturbed wave equation of quantum mechanics (see \cite{PataZel06,GhiMar91,CarChoDlo08})
\begin{align}\label{eq-wave-quantum}
\pdtt u &- \alpha \Delta u - \beta \Delta (\pdt u)  =- |u|^{q} u - |\pdt u|^{p} (\pdt u), \quad p,q \ge 0. 
\end{align}
\end{enumerate}
\medskip

Another type of second-order in time PDE is the Euler--Bernoulli beam equation with Kelvin--Voigt damping 
\begin{subequations} \label{eq-beam-gen}
\begin{align} 
 \pdtt u &+ \pdxx(\alpha \pdxx u  + \beta \pdxx(\pdt u))  + \gamma \pdt u + \delta u = g(u), \quad (x,t) \in (0, L) \times (0, T), \\
u(0,x) &= p(x), \quad \pdt u(0,x) = q(x), \qquad x \in (0,L),
\end{align}
\end{subequations}
where $u(t,x)$ denotes the deflection of the beam of its rigid body motion at time $t$ and position $x$. For given parameters $\alpha>0$ and $\beta \ge 0$, the moment function is 
\begin{align*}
m(t,x) = \alpha \pdxx u(t,x)  + \beta \pdxx(\pdt u)(t,x).
\end{align*}
The first derivative of the moment $m(t,x)$ with respect to the variable $x$ represents the shear force. The following boundary conditions will be considered, where $\xi \in \{0, L\}$:
\begin{enumerate}[(a)]
\item Hinged end: $u(t,\xi) = 0,~~m(t,\xi) = 0$. 
\item Clamped end: $u(t,\xi) = 0,~~~\pdx u(t,\xi) = 0$.
\item Free end: $m(t,\xi) = 0,~~~\pdx m(t,\xi) = 0$. 
\item Sliding end: $\pdx u(t,\xi) = 0,~~~m(t,\xi) = 0$.
\end{enumerate}
Depending on the set up of the beam model, various combinations of boundary conditions are of interest, for example: hinged-hinged boundary conditions
\begin{align*}
u(t,0) = 0,~~m(t,0) = 0,~~u(t,L) = 0,~~m(t,L) = 0.  
\end{align*}
Concerning semilinear beam equations, in \cite{EdaMor19SIAM,EdaMor19,AnsEsmYou11}, a nonlinear term $g(u) = -lu^3,~~ l >0$ was used when the authors considered a railway track model. 

\medskip

Both problems \eqref{eq-wave-gen} and \eqref{eq-beam-gen} can be rewritten as abstract ordinary differential equations in a product space $X = H  \times L^2(\Omega)$ by denoting a new variable $y(t,x) = (u(t,x), w(t,x))'$ as follows 
\begin{align*}
\dot{y}(t) = \cA y(t) + \cF(y(t)),
\end{align*}
where 
\begin{align*}
&\cA = \bmat{0 & I \\ -\alpha (-\Delta) - \delta I & -\beta (-\Delta) - \gamma I} \qquad \text{for \eqref{eq-wave-gen}}, \\
&\cA = \bmat{0 & I \\ -\alpha \pdxxxx  - \delta I & -\beta \pdxxxx  - \gamma I} \qquad \text{for \eqref{eq-beam-gen}}, \\
\end{align*}
and 
\begin{align*}
\mathcal{F}(u,w) = \bmat{0 \\ g(u) + h(w)}.
\end{align*}
The space $H$ and the domain of the operator $\cA$ will be chosen to be consistent with the boundary conditions. Here and henceforth, the transpose of a matrix $E$ is denoted by $E'$.  

\medskip
These types of equations have been studied extensively in many fields  of mathematics. For damped wave equations, see \cite{PataZel06,GhiMar91,CarChoDlo08,ChenFino21,Pon85,IkeTodYor13}; for Euler--Bernoulli beam equations, see\cite{MatSti15,Mor20,BanIto97,EdaMor19SIAM,EdaMor19,PauPhan20,PhanPau21,ItoMor98,LiuLiu98,AnsEsmYou11} and references therein. The time discretization of these equations, to the best of our knowledge, is usually carried out by standard integration schemes such as Runge--Kutta methods or multistep methods. In this article, we will consider exponential integrators to solve this class of PDEs. By spatial discretization of \eqref{eq-wave-gen} or of \eqref{eq-beam-gen}, we get a semi-discretization of the equation in matrix form
\begin{align} \label{eq-abstract-nonlinear1}
\dot{y}(t) = A y(t) + F(y(t)), \qquad y(0) = y_0 =  (p, q)', 
\end{align}
where
\begin{align} \label{eq-optA}
A = \bmat{ 0  & I \\ -\alpha S - \delta I & -\beta S - \gamma  I }.
\end{align}
and the square matrix $S$ is the discretized version of the operator $(-\Delta)$ or $\pdxxxx$. 
The linear part of \eqref{eq-abstract-nonlinear1} 
\begin{align} \label{eq-abstract-linear}
\dot{y}(t) = A y(t) , \qquad y(0) = y_0 = (p, q)', 
\end{align}
can be solved exactly, e.g., 
\begin{align} \label{eq-solexp}
y(t) = \me^{tA} y_0, \qquad t >0 
\end{align}
For the undamped wave equations  (i.e. $\beta = \gamma = 0$ in \eqref{eq-optA}), by using the matrix sine and matrix cosine functions, the explicit form of the matrix exponential  $\me^{tA}$  is easily obtained (see \cite[Section 3.2]{HochOst10}). Based on this formula, Gautschi (in \cite{Gau61}) and Deuflhard  (in \cite{Deu79}) developed a number of schemes to tackle semilinear second-order differential equations. Nevertheless, when damping terms appear in \eqref{eq-optA}, a direct approach to compute the matrix exponential $\me^{tA}$ is more involved and not yet discussed in the literature. Therefore, in this paper, we firstly present an approach to exactly evaluate the matrix exponential of \eqref{eq-optA}. 

\medskip

Let us briefly explain our procedure to compute the matrix exponential. We start by employing two linear transformations to represent the matrix $A$ as $A = \widetilde{Q}PCP' \widetilde{Q}'$ where the new matrix $C$ is a block diagonal matrix, i.e. $C = \diag(G_1,\cdots,G_n)$. Each block $G_i$ is a $2 \times 2$ matrix. The exponential of such a matrix $G_i$ will be computed explicitly. Regarding its eigenvalues, a suitable formula will be constructed. In this way the matrix exponential $\me^{tA}$ can be computed cheaply even for large values of $t$. We also discuss the cases $\beta = \gamma = \delta = 0$ and $\beta, \gamma \ll \alpha$ (see \cite[Section 3]{ItoMor98} for typical physical parameters). In both cases, the matrix $G_i$ has  usually two conjugate complex eigenvalues. To reduce the computation cost, we avoid complex arithmetic. The exact matrix exponential will not only give us a huge advantage to solve the class of linear damped wave equations or linear beam equations but also be valuable in computing solutions of semilinear problems. The numerical schemes for the full equation \eqref{eq-abstract-nonlinear} were constructed by incorporating the exact solution of \eqref{eq-abstract-linear} in an appropriate way. In the literature, these methods were  investigated by many authors (see, e.g., \cite{DoOngThai17,HochOst10,Krog05,StrehWei92,CoxMat02,LuanOst14,HochOst05a,HochOst05b}). To employ these known exponential integrators, the core point is the computation of related matrix functions $\varphi_k(tA)$. As for the matrix exponential, we will use two linear transformations and compute the action of the matrix functions $\varphi_k(tG_i)$. Explicit formulas will be established in the same way as for computing the matrix exponential $\me^{tG_i}$. Concerning the computation of matrix functions, we refer to the review work by Higham and Al-Mohy \cite{HigAl10} as well as the monograph by Higham \cite{Hig08}.

\bigskip 

The outline of the paper is as follows. We start with the discussion of computing the matrix exponential $\me^{tA}$ in section \ref{sec-linear}.  Two linear transformations $P$ and $Q$ will be presented. The computations of the matrix exponential $\me^{tG_i}$ will be discussed for three different cases. In simulations, instead of computing the matrix exponential, we will rather compute its action on a given vector. A detailed instruction will be presented in remark \ref{rem-action}. In section \ref{sec-semilinear} we recall some exponential integrators and discuss an approach to compute the action of the related matrix functions $\varphi_k(tA)$. The procedure will be summarized in section \ref{sec-proc}. In section \ref{sec-numex}, we will present some numerical examples of semilinear equations. The operators $(-\Delta)$ and $\pdxxxx$ will be discretized by finite differences. We will use exponential integrators for the time integration of these examples.  {Some comparisons with standard integrators will be presented in section \ref{sec-comparison} to clarify the efficiency of our approach}.

\section{Exact matrix exponential}  \label{sec-linear}
In this section, we propose an approach to compute efficiently the matrix exponential $\me^{tA}$ for a matrix $A$ of the form \eqref{eq-optA}. With this at hand, the solution of linear system \eqref{eq-abstract-linear} can be evaluated for an arbitrary time  $t>0$ in a fast and reliable way.
\subsection{Two linear transformations}
The key idea is to transform $A$ to a simple block-diagonal matrix for which the exponential can be computed cheaply. 
\begin{lem} \label{lem-trans-Q}
Assume that there exist an orthogonal matrix $Q$ and a diagonal matrix $D = \diag\{\lambda_1,\ldots, \lambda_n\}$ such that $S = QDQ'$, then the matrix $A$ of form \eqref{eq-optA} can be transformed to the block form 
\begin{align} \label{eq-optB}
B = \bmat{0 & I \\ D_1 & D_2},
\end{align}
where $D_1$ and $D_2$ are two diagonal matrices. 
\end{lem}

\begin{proof}
By substituting $S = QDQ'$ and $QQ' = I$ into \eqref{eq-optA}, we get that 
\begin{align*}
A &= \bmat{ 0  & QQ' \\ -\alpha QDQ' - \delta QQ'   & -\beta QDQ' - \gamma QQ'} \\
&= \bmat{Q & 0 \\ 0 & Q} \bmat{0 & I \\ -\alpha D - \delta I & - \beta D - \gamma I  }  \bmat{Q' & 0 \\ 0 & Q'}.
\end{align*}
The proof is complete by identifying two diagonal matrices 
$D_1 = -\alpha D - \delta I$ and  $D_2 = -\beta D - \gamma I$. 
\end{proof}

\medskip
\begin{lem} \label{lem-trans-P}
Consider $P \in \R^{2n \times 2n}$ the permutation matrix satisfying 
\begin{align} \label{eq-matP}
P_{i,2i-1} &= 1, \quad P_{i+n,2i} = 1 \quad \text{~~for~~} 1 \le i \le n, \quad  P_{k,l} = 0  \text{~~else}.  
\end{align}
The matrix $B$ given in \eqref{eq-optB} can be transformed under the permutation $P$ to a block diagonal matrix $C$, i.e.  $B = PCP'$, where
\begin{align*}
C = \bmat {
G_1 & 0 & \cdots & 0 \\
0 & G_2 & \cdots & 0 \\
\vdots & \vdots & \ddots & 0 \\
0 & 0 & \cdots & G_n 
} \quad  \text{with} \quad G_i = \bmat{0 & 1  \\ -\alpha \lambda_i - \delta & -\beta \lambda_i - \gamma}.
\end{align*}
\end{lem}

\begin{proof}
Following the definitions of the matrices $B$ and $C$, for $1 \le i \le n$ we have 
\begin{align*}
B_{i, n+i} &=1,\quad {B_{n+i,i} = -\alpha \lambda_i-\delta},\quad B_{n+i, n+i} =   -\beta \lambda_i -  \gamma, \\
C_{2i-1, 2i} &=1,\quad {C_{2i, 2i-1} = -\alpha \lambda_i-\delta},\quad C_{2i, 2i} =   -\beta \lambda_i - \gamma.
\end{align*}
We will prove that $B = P C P'$. Indeed, for $1 \le i \le n$, we have 
\begin{align*}
(P C P')_{n+i, n+i} &= P_{n+i,2i} (C P')_{2i, n+i} = (C P')_{2i, n+i} = C_{2i,2i}P'_{2i,n+i} \\
&=  C_{2i,2i} P_{i+n,2i} = C_{2i,2i} = -\beta \lambda_i - \gamma = B_{n+i, n+i}, \\
(P C P')_{i, n+i} &= P_{i,2i-1} (C P')_{2i-1, n+i} = (CP')_{2i-1, n+i} = C_{2i-1,2i} P'_{2i,n+i} \\
&=  C_{2i-1,2i}  = 1 =  B_{i,n+i},  \\ 
(P C P')_{n+i, i} &=  P_{n+i,2i} (CP')_{2i, i} = (C P')_{2i, i} = C_{2i,2i-1} P'_{2i-1,i} \\ 
&= C_{2i,2i-1} = -\alpha \lambda_i - \delta = B_{n+i, i}. 
\end{align*}
We will not be concerned with the remaining elements of $B$ and $C$ since they are all zero. Thus, the proof is complete. 
\end{proof}

\medskip
\begin{ex} For $n = 2$ and $n = 3$ the permutation matrices $P$ have the following form 
\begin{align*}
P_2 = \bmat{1 & 0 & 0 &0 \\ 0 & 0 & 1 & 0 \\ 0 & 1 & 0 & 0 \\ 0 & 0 & 0 & 1}, \quad P_3 = \bmat{1 & 0 & 0 & 0 & 0 & 0  \\ 0 & 0 & 1 & 0 & 0 & 0 \\ 0 & 0 & 0 &0 & 1 & 0  \\ 0 & 1 & 0 & 0 & 0 &0  \\ 0 &0 &0 & 1 & 0 &0 \\ 0 & 0 & 0 & 0 & 0 & 1}.
\end{align*}
\end{ex}
Next, we recall some important properties of matrix functions (see \cite[Theorem 1.13]{Hig08} or \cite[Theorem 2.3]{HigAl10}).
\begin{theorem} \label{the-propmatrixcomputations}
Let $A \in \C^{n \times n}$ and $f$ be defined on the spectrum of $A$. Then 
\begin{enumerate}[(a)]
\item $f(A)$ commutes with $A$; 
\item $f(A') = f(A)' $; 
\item $f(XAX^{-1}) = Xf(A) X^{-1}$.
\item The eigenvalues of $f(A)$ are $f(\lambda_i)$, where $\lambda_i$ are the eigenvalues of $A$.
\item If $A = (A_{ij})$ is block triangular then $F=f(A)$ is block triangular with the same block structure as $A$, and $F_{ii} = f(A_{ii})$. 
\item If $A = \diag(A_{11}, A_{22}, \dots, A_{mm})$ is block diagonal then 
\begin{align*}
f(A) = \diag(f(A_{11}), f(A_{22}), \dots, f(A_{mm})).
\end{align*}
\end{enumerate}
\end{theorem}
A direct consequence of this theorem is the following result. 
\begin{theorem} \label{the-exp}
Assume that there exist an orthogonal matrix $Q$ and a diagonal matrix $D = \diag\{\lambda_1, \cdots, \lambda_n\}$ such that $S = Q D Q'$.
Then, for $t >0$, the exponential of the matrix $tA$ is computed as follows 
\begin{align} \label{eq-exp-A}
\me^{tA} = \bmat{Q & 0 \\ 0 & Q} P 
 \bmat {
\me^{tG_1} & 0 & \cdots & 0 \\
0 & \me^{tG_2} & \cdots & 0 \\
\vdots & \vdots & \ddots & 0 \\
0 & 0 & \cdots & \me^{tG_n} 
} P' \bmat{Q' & 0 \\ 0 & Q'},
\end{align}
where $P \in \R^{2n \times 2n}$ is defined by \eqref{eq-matP} and $G_i = \bmat{0 & 1  \\ -\alpha \lambda_i - \delta & -\beta \lambda_i - \gamma}$. 
\end{theorem}
\begin{proof}
The two lemmas \ref{lem-trans-Q} and \ref{lem-trans-P} imply that
\begin{align*}
A = \bmat{Q & 0 \\ 0 & Q} B  \bmat{Q' & 0 \\ 0 & Q'}
= \bmat{Q & 0 \\ 0 & Q} P 
 \bmat {
{G_1} & 0 & \cdots & 0 \\
0 & {G_2} & \cdots & 0 \\
\vdots & \vdots & \ddots & 0 \\
0 & 0 & \cdots & {G_n} 
} P' \bmat{Q' & 0 \\ 0 & Q'}.  
\end{align*}
Formula \eqref{eq-exp-A} is proved by using the properties (c) and (f) in Theorem \ref{the-propmatrixcomputations}. 
\end{proof}
\begin{rem}
We need to compute the matrix exponential of the small matrices $tG_i$. This will be presented in the next section. The exponential matrix $\me^{tG_i}$ can be computed explicitly by using formula \eqref{eq-exp-G-1}, \eqref{eq-exp-G-2}, or \eqref{eq-exp-G-3} depending on the sign of $(\beta \lambda_i + \gamma)^2-4(\alpha \lambda_i + \delta)$.   
\end{rem}
\begin{rem} \label{rem-action}
In practical situations, {to reduce the computational cost}, we will compute the action of the matrix exponential to a vector instead of computing it explicitly. In \eqref{eq-exp-A}, $P$ and $Q$ are two square matrices of orders $2n$ and $n$, respectively. Since $P$ is a permutation matrix, it can be stored, however,  as a column matrix with $n$ entries by indicating the positions of the non-zero elements, for example:
\begin{align*}
P_3 = \bmat{1 & 0 & 0 & 0 & 0 & 0  \\ 0 & 0 & 1 & 0 & 0 & 0 \\ 0 & 0 & 0 &0 & 1 & 0  \\ 0 & 1 & 0 & 0 & 0 &0  \\ 0 &0 &0 & 1 & 0 &0 \\ 0 & 0 & 0 & 0 & 0 & 1} \rightarrow \bmat{1 \\ 3 \\ 5 \\ 2 \\ 4 \\ 6},\quad P'_3 = \bmat{1 & 0 & 0 & 0 & 0 & 0  \\ 0 & 0 & 0 & 1 & 0 & 0 \\ 0 & 1 & 0 &0 & 0 & 0  \\ 0 & 0 & 0 & 0 & 1 &0  \\ 0 & 0 & 1 & 0 & 0 &0 \\ 0 & 0 & 0 & 0 & 0 & 1} \rightarrow \bmat{1 \\ 4 \\ 2 \\ 5 \\ 3 \\ 6}
\end{align*}
The block matrix $[\me^{tG_i}]$ can be stored as a $2 \times 2n$ matrix. Given a compound vector $v_0 = \bmat{a, b}'$, where $a$ and $b$ are two column vectors with $n$ entries, we start by evaluating a new vector $v_1 = [Q'a, Q'b]'$. Next, the action of the permutation matrix $P'$ to the vector $v_1$ is the reorder of its entries to get a new vector $v_2$. Then we compute the multiplication of the block exponential matrix with the vector $v_3$ by cheaply multiplying each block $2 \times 2$ matrix $\me^{tG_i}$ with two corresponding elements of $v_2$. Analogously applying $P$ and $Q$, we get an exact valuation of the action of the matrix exponential to an arbitrary vector.
\end{rem}

\subsection{The matrix exponential $G_i$}
From \eqref{eq-exp-A}, instead of evaluating the matrix exponential of $A \in \R^{2n \times 2n}$, we need to compute the matrix exponential of each $G_i \in \R^{2 \times 2}$. In this section, we give some explicit formulas. For simplification, we omit the index $i$. 
\begin{theorem} \label{the-funcG}
Assume that $f$ is an analytic function. For a $2 \times 2$ matrix $G$, the matrix function $f(G)$ can be computed explicitly as
\begin{align} \label{eq-funcG}
f(G) = \frac{f(z_1) - f(z_2)}{z_1-z_2} G + \frac{z_1 f(z_2) - z_2 f(z_1)}{z_1-z_2} I,
\end{align}
where $z_1$ and $z_2$ are the two distinct eigenvalues of the matrix $G$.
In case the matrix $G$ has a double eigenvalue $z_1$, we get 
\begin{align} \label{eq-funcG-repeig}
f(G) = f'(z_1) G + \left(f(z_1) - f'(z_1) z_1 \right) I.  
\end{align} 
\end{theorem} 
\begin{proof}
Let $p(z)$ be the characteristic polynomial of the matrix $G$ and assume for a moment that the equation $p(z) = 0$ has two distinct roots $z_1$ and $z_2$. The Cayley--Hamilton theorem states that $p(G) = 0$. \\
The function $f$ can be rewritten in the form $f(z) = q(z) p(z) + r(z)$ where $q(z)$ is some quotient and $r(z)$ is a remainder polynomial with $0 \le \deg r(z) <2$. From $p(G) = 0$, we obtain 
\begin{align*}
f(G) = r(G) =  d_1 G + d_0 I. 
\end{align*}
To complete the proof, we determine the coefficients $d_1$ and $d_0$. From $f(z_1) = r(z_1)$ and $f(z_2) = r(z_2)$, we obtain that 
\begin{align*}
d_1 = \frac{f(z_1) - f(z_2)}{z_1 - z_2}, \qquad  {d_0} = \frac{z_1 f(z_2) - z_2 f(z_1)}{z_1-z_2}. 
\end{align*}
In case of a double eigenvalue $z_1$, we use the conditions $f(z_1) = r(z_1)$ and $f'(z_1) = r'(z_1)$. As a consequence, we obtain that $d_1 = f'(z_1)$ and $d_0 = f(z_1) - f'(z_1) z_1$. 
\end{proof}
We remark that similar formulas can be found in the work of Bernstein and So \cite{BerSo93} or Cheng and Yau \cite{ChengYau97}. To reduce the computational cost, we try to avoid complex arithmetic. 
\begin{lem} \label{lem-exp-G}
Assume that the matrix $G$ is of the form 
\begin{align*}
G =  \bmat{0 & 1 \\ -\alpha \lambda - \delta & -\beta \lambda - \gamma}
\end{align*}
and denote $m = -\frac{1}{2} (\beta \lambda + \gamma )$. 
\begin{enumerate}[(i)]
\item If $(\beta \lambda +  \gamma)^2 > 4 (\alpha \lambda + \delta)$, denoting $n = \frac{1}{2} \sqrt{(\beta \lambda + \gamma)^2-4 (\alpha \lambda+ \delta)  } $, the exponential matrix $\me^{tG}$ can be computed explicitly as follows 
\begin{align} \label{eq-exp-G-1}
\me^{tG} = \frac{\me^{t(m+n)}-\me^{t(m-n)}}{2n}\bmat{
-m-n & 1 \\ n^2 - m^2 & m-n 
} + \me^{t(m+n)} I.
\end{align}
\item If $(\beta \lambda +  \gamma)^2 = 4 (\alpha \lambda + \delta)$, we obtain that 
\begin{align} \label{eq-exp-G-2}
\me^{tG} = \me^{tm} \bmat{1-tm & t \\ -tm^2 & tm+1}. 
\end{align}
\item If $(\beta \lambda +  \gamma)^2 < 4 (\alpha \lambda + \delta)$, denoting  $n = \frac{1}{2} \sqrt{4 (\alpha \lambda + \delta) - (\beta \lambda + \gamma)^2 } $, we get that 
\begin{align} \label{eq-exp-G-3}
\me^{tG} =  \frac{\me^{tm}\sin (tn)}{n} \bmat{-m & 1 \\ -n^2 - m^2 & m} + e^{tm} \cos (tn) I. 
\end{align}
\end{enumerate}
\end{lem}

\begin{proof}
Let $z_1$ and $z_2$ be the two eigenvalues of the matrix $tG$. Thus $z_1$ and $z_2$ satisfy the characteristic equation $z^2 + (\beta \lambda + \gamma) t z + (\alpha \lambda + \delta) t^2 = 0$. By using formula \eqref{eq-funcG}, we obtain that 
\begin{align} \label{eq-exp-G-gen}
\me^{tG} = \frac{\me^{z_1}-\me^{z_2}}{z_1 -  z_2} tG + \frac{z_1 \me^{z_2} - z_2 \me^{z_1}}{z_1 - z_2} I.
\end{align}
The discriminant of the characteristic equation is 
\begin{align*}
D = \left((\beta \lambda + \gamma)^2 - 4(\alpha \lambda + \delta) \right) t^2.
\end{align*}
We consider three cases:
\begin{enumerate}[(i)]
\item If $D > 0$ or $(\beta \lambda + \gamma)^2 > 4(\alpha \lambda + \delta)$, the two real roots of the characteristic equation are $z_1 = tm + tn$ and $z_2 = tm - tn$, where  $m = -\frac{1}{2} (\beta \lambda + \gamma )$ and $n = \frac{1}{2} \sqrt{(\beta \lambda + \gamma)^2 - 4 (\alpha \lambda +\delta)}$. From the definitions of the parameters $m$ and $n$, we get that $-\alpha \lambda - \delta = n^2-m^2$ and $-\beta \lambda - \gamma = 2m$. We will simplify the two coefficients in formula \eqref{eq-exp-G-gen}:
\begin{align*}
\frac{\me^{z_1}-\me^{z_2}}{z_1 -  z_2} &= \frac{\me^{t(m+n)} - \me^{t(m-n)} }{2tn} , \\
\frac{z_1 \me^{z_2} - z_2 \me^{z_1}}{z_1 - z_2} &= \frac{(m+n)\me^{t(m-n)}-(m-n) \me^{t(m+n)}}{2n}  \\
&= \me^{t(m+n)}  - (m+n) \frac{\me^{t(m+n)} - \me^{t(m-n)} }{2n}.
\end{align*}
By substituting the two simplified coefficients into \eqref{eq-exp-G-gen}, we get that
\begin{align*}
\me^{tG} &= \frac{\me^{t(m+n)} - \me^{t(m-n)} }{2n} \left(G - (m+n) I \right) + \me^{t(m+n)} I \\
&= \frac{\me^{t(m+n)} - \me^{t(m-n)} }{2n} \left( 
\bmat{0 & 1 \\ n^2-m^2 & 2m} - \bmat{m+n & 0 \\ 0 & m+n} 
\right)  + \me^{t(m+n)} I \\
&= \frac{\me^{t(m+n)} - \me^{t(m-n)} }{2n} \bmat{-m-n & 1 \\ n^2-m^2 & m-n } + \me^{t(m+n)} I. 
\end{align*} 
\bigskip
\item If $D =0$ or $(\beta \lambda + \gamma)^2 = 4(\alpha \lambda + \delta)$, the characteristic equation has only one root $z_1 = tm$, where $m = -\frac{1}{2}(\beta \lambda + \gamma)$. In this case, we have 
\begin{align*}
\me^{tG} &= \me^{z_1} (tG) + \me^{z_1} (1-z_1) I  \\
&= \me^{tm} t \bmat{0 & 1 \\ -m^2 & 2m} + \me^{tm} (1-tm)\bmat{1 & 0 \\ 0 & 1} = \me^{tm} \bmat{1-tm & t \\ -tm^2 & tm+1}. 
\end{align*}
\bigskip

\item If  $D <0$ or $(\beta \lambda + \gamma)^2 < 4(\alpha \lambda + \delta)$, the characteristic equation has two conjugate complex roots $z_1 = tm +\mi tn$ and $z_2 = tm - \mi tn$, where 
\begin{align*}
 m = -\frac{1}{2} (\beta \lambda + \gamma ), \qquad n = \frac{1}{2} \sqrt{4 (\alpha \lambda + \delta) - (\beta \lambda + \gamma)^2 }. 
\end{align*}
Since $z_2 = \conj{z_1}$, we infer that $\me^{z_2} = \conj{\me^{z_1}}$ and $z_2 \me^{z_1} = \conj{z_1} \conj{\me^{z_2}} = \conj{z_1\me^{z_2}}$. We analogously simplify the two coefficients in \eqref{eq-exp-G-gen} as follows 
\begin{align*}
\frac{\me^{z_1}-\me^{z_2}}{z_1 -  z_2} &= \frac{\me^{z_1} - \overline{\me^{z_1}}}{z_1 - \overline{z_1}} = \frac{2\im(\me^{z_1})}{2\im(z_1)} =  \frac{\me^{tm} \sin (tn) }{tn} , \\
\frac{z_1 \me^{z_2} - z_2 \me^{z_1}}{z_1 - z_2} &= \frac{\im(z_1 \conj{\me^{z_1}})}{\im(z_1)} = \me^{tm} \cos (tn) - m \frac{\me^{tm}\sin (tn)}{n}  . 
\end{align*}
By substituting the two simplified coefficients into \eqref{eq-exp-G-gen} and noting that $-\alpha \lambda - \delta = -n^2-m^2$ and $-\beta \lambda - \gamma = 2m$, we get that
\begin{align*}
\me^{tG} &= \frac{\me^{tm}\sin (tn)}{n} (G- m I) + e^{tm} \cos (tn) I \\
& = \frac{\me^{tm}\sin (tn)}{n} \bmat{-m & 1 \\ -n^2 - m^2 & m} + e^{tm} \cos (tn) I. 
\end{align*} 
\end{enumerate}
This concludes the proof. 
\end{proof}

\begin{rem}
Formula \eqref{eq-exp-G-3} is useful in computations. For example, for beam equations with typical physical parameters proposed by Ito and Morris in \cite[Section 3]{ItoMor98}, the matrix $G$ has usually two complex conjugate eigenvalues.
\end{rem}

\section{Exponential integrators} \label{sec-semilinear}
\subsection{Exponential integrators for semilinear problems}
We consider semilinear differential equations of the form 
\begin{align} \label{eq-abstract-nonlinear}
\dot{y}(t) = A y(t) + F(y(t)).
\end{align}
The solution of this equation at time $t_{n+1} = t_n + \tau_n,~~~t_0 = 0, n \in \N$ is given by the variation-of-constants formula 
\begin{align*}
y(t_{n+1}) = \me^{\tau_n A} y(t_n) + \int_{0}^{\tau_n}  \me^{(\tau_n  - \tau)A} F(y(t_n + \tau)) d\tau.  
\end{align*} 
For the numerical soltuion of \eqref{eq-abstract-nonlinear}, we recall a general class of one-step exponential integrators from \cite{HochOst05a,HochOst05b,HochOst10}
\begin{align*}
y_{n+1} &= \me^{\tau_n A} y_n + \tau_n \sum_{i=1}^s b_i(\tau_n A) F_{ni}, \\
Y_{ni} &=  \me^{c_j \tau_n A} y_n + \tau_n \sum_{j=1}^{i-1} a_{ij} (\tau_n A) F_{nj},  \\
F_{nj} &= F(Y_{nj}).  
\end{align*}
The coefficients are as usually collected in a Butcher tableau
\begin{center}
\begin{tabular}{c|c c c c}
$c_1$ & \\ 
$c_2$ & $a_{21}(\tau_n A)$ \\
$\vdots$ & $\vdots$ & $\ddots$\\
$c_s$ & $a_{s1}(\tau_n A)$  & $\hdots$ & $a_{s,s-1}(\tau_n A)$        \\ [0.25em]
\hline 
\\ [-0.75em]
& $b_1(\tau_n A)$ & $\hdots$ & $b_{s-1}(\tau_n A)$ & $b_{s}(\tau_n A)$ 
\end{tabular}
\end{center}
\medskip
The method coefficients $a_{ij}$ and $b_i$ are constructed from a family of functions $\varphi_k$ evaluated at the matrix $(\tau_n A)$. We next recall this family $\varphi_k$, which was introduced before in \cite{HochOst05a,HochOst05b,HochOst10}. 
\begin{cor} Consider the entire functions
\begin{align*}
\varphi_k(z) = \int_0^1 \me^{(1-\theta)z} \frac{\theta^{k-1}}{(k-1)!} d\theta, \qquad k \ge 1, \qquad \varphi_0(z) = \me^z.
\end{align*}
These functions satisfy the following properties: 
\begin{enumerate}[(i)]
\item $\displaystyle \varphi_k(0) = \frac{1}{k!} $;
\item they satisfy the recurrence relation 
\begin{align*}
\varphi_{k+1}(z) = \frac{\varphi_{k}(z) - \varphi(0)}{z};
\end{align*}
\item the Taylor expansion of the function $\varphi_k$ is 
\begin{align*}
\varphi_k (z) = \sum_{n=0}^\infty \frac{z^n}{(n+k)!}. 
\end{align*}
\end{enumerate}
\end{cor}
 \noindent To simplify notation, we denote 
\begin{align*}
\varphi_{k,j} = \varphi_k( c_j \tau_n A), \quad \varphi_k = \varphi_k(\tau_n A). 
\end{align*} 
Next, we recall five exponential integrators that will be used in our numerical examples. 
\begin{ex} For $s=1$, the exponential Euler method has the form 
\begin{align} \label{eq-exp-Euler}
y_{n+1} = \me^{\tau_n A} y_n + \tau_n \varphi_1 (\tau_n A) F(y_n).
\end{align} 
\begin{center} \ttfamily{(EI-E1)} \qquad 
\begin{tabular}{c|c}
$0$ & \\ 
\hline 
& $\varphi_1$ 
\end{tabular}
\end{center} 
We denote this method by \ttfamily{EI-E1}.
\end{ex}

\begin{ex} For $s=2$, we recall a second-order method proposed by Strehmel and Weiner in \cite[Section 4.5.3]{StrehWei92}: 
\begin{center} \ttfamily{(EI-SW21)} \qquad 
\begin{tabular}{c|c c }
$0$ & \\
$c_2$ & $c_2 \varphi_{1,2}$ \\ [0.25em]
\hline 
\\ [-0.75em]
& $\varphi_1 - \frac{1}{c_2} \varphi_2$ & $\frac{1}{c_2} \varphi_2$  
\end{tabular}
\end{center} 
A simplified version, where only $\varphi_1$ is used, is also proposed by Strehmel and Weiner
\begin{center} \ttfamily{(EI-SW22)} \qquad 
\begin{tabular}{c|c c }
$0$ & \\
$c_2$ & $c_2 \varphi_{1,2}$ \\ [0.25em]
\hline 
\\ [-0.75em]
& $\left( 1 - \frac{1}{2c_2} \right) \varphi_1$ & ${\frac{1}{2c_2}} \varphi_1$  
\end{tabular}
\end{center} 
\end{ex}

\begin{ex} For $s=4$, we recall two schemes. The first one, proposed by Krogstad in \cite{Krog05}, is given by 
\begin{center} \ttfamily{(EI-K4)} \qquad 
\begin{tabular}{c|c c c c }
$0$ & \\ 
$\frac{1}{2}$ & $\frac{1}{2} \varphi_{1,2}$  \\ [0.5em]
$\frac{1}{2}$ & $\frac{1}{2} \varphi_{1,3} - \varphi_{2,3} $ & $\varphi_{2,3} $ \\ [0.5em]
$1$ & $\varphi_{1,4} - 2\varphi_{2,4} $ & $0$ & $2\varphi_{2,4} $ \\ [0.25em]
\hline 
\\ [-0.75em]
&  $\varphi_1 - 3 \varphi_2 + 4\varphi_3 $ & $2 \varphi_2 - 4\varphi_3$ &  $2 \varphi_2 - 4\varphi_3$  &  $-\varphi_2 + 4\varphi_3$ 
\end{tabular}
\end{center} 
The second method is suggested by Strehmel and Weiner (see \cite[Example 4.5.5]{StrehWei92}) 
\begin{center} \ttfamily{(EI-SW4)} \qquad 
\begin{tabular}{c|c c c c }
$0$ & \\ 
$\frac{1}{2}$ & $\frac{1}{2} \varphi_{1,2}$  \\ [0.5em]
$\frac{1}{2}$ & $\frac{1}{2} \varphi_{1,3} - \frac{1}{2} \varphi_{2,3} $ & $\frac{1}{2} \varphi_{2,3} $ \\ [0.5em]
$1$ & $\varphi_{1,4} - 2\varphi_{2,4} $ & $- 2\varphi_{2,4}$ & $4\varphi_{2,4} $ \\  [0.25em]
\hline 
\\ [-0.75em]
&  $\varphi_1 - 3 \varphi_2 + 4\varphi_3 $ & $0$ &  $4 \varphi_2 - 8\varphi_3$  &  $-\varphi_2 + 4\varphi_3$ 
\end{tabular}
\end{center} 
\end{ex}

\bigskip

\subsection{Computing matrix functions of $tA$}
To apply these exponential integrators to semilinear problems, we next introduce an approach to explicitly compute the matrix functions $\varphi_k(tA)$. We first present an analogous version of Theorem \ref{the-exp}.
\begin{theorem} \label{the-phik}
Assume that there exist an orthogonal matrix $Q$ and a diagonal matrix $D = \diag\{\lambda_1, \cdots, \lambda_n\}$ such that $S = Q D Q'$.
Then, for $t >0$ and $k \ge 1$, the functions $\varphi_k(tA)$ are computed as follows 
\begin{align} \label{eq-varphi-A}
\varphi_k(tA) = \bmat{Q & 0 \\ 0 & Q} P 
 \bmat {
\varphi_k(tG_1) & 0 & \cdots & 0 \\
0 & \varphi_k(tG_2)& \cdots & 0 \\
\vdots & \vdots & \ddots & 0 \\
0 & 0 & \cdots &\varphi_k(tG_n) 
} P' \bmat{Q' & 0 \\ 0 & Q'}, 
\end{align}
where $P \in \R^{2n \times 2n}$ is given in \eqref{eq-matP} and $G_i = \bmat{0 & 1  \\ -\alpha \lambda_i - \delta & -\beta \lambda_i - \gamma}$. 
\end{theorem} 
The matrix functions $\varphi_k(tG_i)$ are computed explicitly. The actual formula depends on the sign of $(\beta \lambda_i + \gamma)^2-4(\alpha \lambda_i + \delta)$. Next, we will present two lemmas concerning these functions.      
\begin{lem} Assume that the matrix $G$ is of the form 
\begin{align*}
G =  \bmat{0 & 1 \\ -\alpha \lambda - \delta & -\beta \lambda - \gamma}
\end{align*}
and denote $m = -\frac{1}{2} (\beta \lambda + \gamma )$. 
\begin{enumerate}[(i)]
\item If $(\beta \lambda +  \gamma)^2 > 4 (\alpha \lambda + \delta)$, denoting $n = \frac{1}{2} \sqrt{(\beta \lambda + \gamma)^2-4 (\alpha \lambda+ \delta)  } $, the matrix functions $\varphi_k(tG)$ can be computed explicitly as follows 
\begin{align} \label{eq-varphi-G-1}
\varphi_k(tG) = \frac{\varphi_k^+ -\varphi_k^-}{2n}\bmat{
-m-n & 1 \\ n^2 - m^2 & m-n 
} + \varphi_k^+ I,
\end{align}
where $\varphi_k^+ = \varphi_k(t(m+n))$ and $\varphi_k^- = \varphi_k(t(m-n))$.
\item If $(\beta \lambda +  \gamma)^2 = 4 (\alpha \lambda + \delta)$, we obtain that 
\begin{align} \label{eq-varphi-G-2}
\varphi_k(tG) = \varphi'_k(tm) \bmat{-tm & t \\ -tm^2 & tm} + \varphi_k(tm) I, 
\end{align}
where the derivative $\varphi^\prime_k(z)$ can be computed recursively 
\begin{align*}
\varphi_0(z) &= \varphi_0^\prime(z) = \me^z, \\
\varphi^\prime_{k+1}(z) &=  \frac{\varphi_k'(z) - \varphi_{k+1}(z)}{z}, 
\quad \varphi_{k+1}(z) = \frac{\varphi_k(z) - \varphi_{k}(0)}{z}.
\end{align*}
\end{enumerate}
\end{lem}
\begin{proof}
By applying Theorem \ref{the-funcG}, the proof follows the lines of the first two points in the proof of Lemma \ref{lem-exp-G}. 
\end{proof}
The last lemma concentrates on the case of two complex eigenvalues of $G_i$. Again the idea is to compute the matrix functions without explicitly using complex numbers. It is inspired by formula \eqref{eq-exp-G-3} above. 
\begin{lem}
In the case $(\beta \lambda +  \gamma)^2 < 4 (\alpha \lambda + \delta)$, the matrix $G$ has two conjugate complex eigenvalues $z_1$ und $z_2$ with $z_1 = tm + i tn$, where 
\begin{align*}
 m = -\frac{1}{2} (\beta \lambda + \gamma ), \qquad n = \frac{1}{2} \sqrt{4 (\alpha \lambda + \delta) - (\beta \lambda + \gamma)^2 }.
\end{align*}
The matrix $\varphi_k(tG)$ can be explicitly computed as follows
\begin{align}\label{eq-varphi-G-3}
\varphi_k(tG) = \frac{i_k}{n} \bmat{-m & 1 \\ -n^2 - m^2 & m} + r_k I. 
\end{align}
Here, the two coefficients $i_k$ and $r_k$ depend on $(t,m,n)$ and can be computed recursively as follows 
\begin{subequations} 
\begin{align}
i_0 &=\me^{tm} \sin(tn), \quad r_0 =  \me^{tm} \cos (tn), \\
i_{k} &= \frac{1}{t(m^2 + n^2)} \left( m i_{k-1} - n  \left(r_{k-1}- \frac{1}{(k-1)!} \right) \right), \\
r_{k} &=  \frac{1}{t(m^2 + n^2)} \left( n i_{k-1} + m \left(r_{k-1}- \frac{1}{(k-1)!} \right) \right).
\end{align}
\end{subequations} 
\end{lem}

\begin{proof}
By using formula \eqref{eq-funcG}, we obtain that
\begin{align*}
\varphi_k(tG) = \frac{\varphi_k(z_1)-\varphi_k(z_2) }{z_1 - z_2} tG +  \frac{z_1\varphi_k(z_2)-z_2\varphi_k(z_1) }{z_1 - z_2} I. 
\end{align*}
First, we note that $\varphi_{k}(z_2) = \conj{\varphi_{k}(z_1)}$ because $\varphi_k$ has real coefficients.
Thus we can simplify as follows 
\begin{align*}
 \frac{\varphi_k(z_1)-\varphi_k(z_2) }{z_1 - z_2}  = \frac{\im(\varphi_k(z_1))}{\im z_1}, \quad \frac{z_1\varphi_k(z_2)-z_2\varphi_k(z_1) }{z_1 - z_2} = \frac{\im(z_1\conj{\varphi_k(z_1)})}{\im z_1}. 
\end{align*}
Next, we rewrite the recursion as follows
\begin{align*}
\varphi_{k+1}(z_1) = \frac{\varphi_k(z_1)- \varphi_k(0)}{z_1} = \frac{(\varphi_k(z_1)- \varphi_k(0)) \conj{z_1}}{|z_1|^2}.
\end{align*} 
To simplify notation, we denote $i_{k} = \im(\varphi_{k} (z_1))$ and  $r_{k} = \re(\varphi_{k} (z_1))$. Thus we obtain that 
\begin{align*}
i_{k+1} &= \im (\varphi_{k+1}(z_1))  = \frac{1}{|z_1|^2} \im \Big((\varphi_k(z_1) - \varphi_k(0)) \conj{z_1} \Big) \\
&= \frac{1}{t^2(m^2+n^2)}\Big( \im (\varphi_k(z_1) - \varphi_k(0)) \re(\conj{z_1}) +  \re (\varphi_k(z_1) - \varphi_k(0)) \im(\conj{z_1}) \Big)\\
&= \frac{1}{t(m^2+n^2)}\left( m i_k  - n \left(r_k - \frac{1}{k!}\right) \right), \\
r_{k+1} &= \re (\varphi_{k+1}(z_1)) = \frac{1}{|z_1|^2} \re \Big((\varphi_k(z_1) - \varphi_k(0)) \conj{z_1} \Big) \\
&=\frac{1}{t^2(m^2+n^2)} \Big( \re (\varphi_k(z_1) - \varphi_k(0)) \re(\conj{z_1}) -  \im (\varphi_k(z_1) - \varphi_k(0)) \im(\conj{z_1}) \Big) \\
&= \frac{1}{t(m^2+n^2)} \left( n i_k + m \left( r_k - \frac{1}{k!} \right) \right). 
\end{align*}
Besides, we also get that 
$
\im(z_1 \conj{\varphi_{k}(z_1)}) = tn r_k - tm i_k 
$. This finally yields that 
\begin{align*}
\varphi_k(tG) &= \frac{i_k}{n} G +  \frac{nr_k - mi_k}{n} I = \frac{i_k}{n} (G-m I) + r_k I \\
&= \frac{i_k}{n} \bmat{-m & 1 \\ -n^2 - m^2 & m} + r_k I,
\end{align*}
which completes the proof. 
\end{proof}
\subsection{Summary of the integration procedure} \label{sec-proc}
The above described procedure can be summarized in two main parts. The \textbf{P}repartion part, which is done once at the beginning, consists of three steps: 
\begin{enumerate}[\bfseries P1:]
\item Discretize the operator $(-\Delta)$ or $\pdxxxx$ as a square symmetric matrix $S$ (e.g., by finite differences, see also section \ref{sec-numex}).
\item Find an orthogonal matrix $Q$ and a diagonal matrix $D = \diag\{\lambda_1, \dots, \lambda_n \}$ such that $S = Q D Q'$. The matrix $D$ is stored as a vector. 
\item Create a column vector which stores the positions of all non-zero entries of the permutation matrix $P$ by using formula \eqref{eq-matP}.
\end{enumerate}
\medskip
The \textbf{M}ain part is used to compute the action of a matrix functions. This is required in the time stepping. Computing this action consists of two steps: 
\begin{enumerate}[\bfseries M1:]
\item Compute the matrix functions $\varphi_k(tG_i)$ by using formulas \eqref{eq-exp-G-1}, \eqref{eq-exp-G-2}, or \eqref{eq-exp-G-3} for $\varphi_0(tG_i) = \me^{tG_i}$; formulas \eqref{eq-varphi-G-1}, \eqref{eq-varphi-G-2}, or \eqref{eq-varphi-G-3} for $\varphi_k(tG_i)$ with $k \ge 1$.  
\item Compute the action of the matrix functions $\varphi_k(tA)$ using formula \eqref{eq-exp-A} for $k = 0$ and formula \eqref{eq-varphi-A} for $k \ge 1$.      
\end{enumerate}

\section{Numerical examples} \label{sec-numex}
\subsection{Semilinear wave equations}
We consider a 1D semilinear wave equation on $\Omega = (0,\ell)$
\begin{subequations}  \label{eq-example-wave}
\begin{align}
\pdtt u &- \alpha \pdxx u - \beta \pdxxt u  + \delta u + \gamma \pdt u = {g(u) + h(\pdt u)},  \quad 0 < x < \ell,~~t \in (0,T], \\
u(t,0) &= 0, \quad u(t,\ell) = 0, \\
u(0,x) &= p(x), \quad u_t(0,x) = q(x). 
\end{align}
\end{subequations}
We consider the product space $X = H^1_0(\Omega) \times L^2(\Omega)$ and rewrite \eqref{eq-example-wave} in abstract form 
\begin{align*}
\dot{y}(t) = \cA y(t)+ \cF(y(t)), \qquad y(0)= (p, q)', 
\end{align*}
where $\cA: D(\cA) \to X$ is the operator defined by
\begin{align*}
\cA \bmat{u \\ w} &= \bmat{0 & I \\ -\alpha (-\pdxx) -  \delta I & - \beta (-\pdxx)- \gamma I }\bmat{u \\ w} \\
&= \bmat{w \\  -\alpha (-\pdxx u)  - \beta (-\pdxx w) - \delta u- \gamma w }, \\
D(\cA) &= \left( H^2(\Omega) \cap H_0^1(\Omega) \right)^2 , \\
\cF \pmat{u \\ w} &= \bmat{0 \\ {g(u) + h(w)}}.
\end{align*}
Define the closed self-adjoint positive operator $\cA_0$ on $L^2(0,\ell)$ by
\begin{align*}
\cA_0 \phi &= -\pdxx \phi, \\
D(\cA_0) &= \{ \phi \in H^2(0,\ell) \mid \phi(0) = \phi(1) =0 \}.
\end{align*}
We use symmetric finite differences to discretize the operator $\cA_0$. For this, the space interval $(0,\ell)$ is divided equidistantly by the nodes $x_i = i \Delta x,~~i \in \{0, \dots, N+1 \}$, where $N$ is a given integer and $\Delta x = \frac{\ell}{N+1}$. Then, the discrete operator $\cA_0$ is given by the matrix $S_w \in \R^{N \times N}$ defined by 
\begin{align} \label{eq-discrete-Sw}
S_w = \frac{1}{\Delta x^2}\bmat{
2 & -1 & 0 & \ldots & 0 \\
-1 & 2 & -1 & \ddots & \vdots \\
0 & \ddots & \ddots &  \ddots & 0 \\
\vdots & \ddots & -1 & 2 & -1 \\
0 & \ldots  & 0 & -1 & 2 
}. 
\end{align}
In the four examples below, we consider the space interval $\Omega =(0,1)$ with $N = 200$. 
\begin{ex}\label{ex-wave1}
Consider equation \eqref{eq-example-wave} with $\alpha =\pi^2,~\beta = 10^{-2},~\delta = 0,~\gamma = 10^{-2}$. The nonlinear source term is $g(u) = \sin u$. This is a perturbed sine-Gordon equation of the form \eqref{eq-wave-sineGordon}. The initial conditions are $p(x) = 5 \sin(2\pi x)$ and $q(x) = 0$. We use {four} different schemes, namely  {\ttfamily{EI-E1}}, {\ttfamily{EI-SW21}} ({with $c_2 = 0.75$}), {\ttfamily{EI-SW4}}, and {\ttfamily{EI-K4}} to compute the solution at time $T = 6$ with $M \in \{5, 10, 20, \cdots, 5 \cdot 2^{11}\}$ time steps. {The reference solution $y_{ref} = (u_{ref}, (\pdt u)_{ref})'$ plotted in Figures \ref{fig-ex1-refsol-u} and \ref{fig-ex1-refsol-ut} is computed by using {\ttfamily{EI-SW4}} with $M = 200000$ time steps. The discrete $\ell_2$ error between the approximate solution obtained with the mentioned integrators at the final time $y(T) = (u(T), \pdt u(T))'$ and the reference solution $y_{ref}$ is computed by the formula
\begin{align} \label{eq-dis-L2-err}
\| y(T) - y_{ref} \|^2_{\ell_2} = \Delta x \sum_{i=1}^N |y_i(T) -y_{ref,i}|^2. 
\end{align}}
These errors are plotted in Figure \ref{fig-ex1-rate}. The expected convergence rate is observed for each scheme. Even when we use a rather coarse time mesh with $M = 5$ and $\Delta t= 1.2$, the error is quite small (approximate $10^{-1}$). 
\begin{figure}[h]
\centering 
\begin{subfigure}[]{0.495\textwidth}
         \centering
         \includegraphics[width=\textwidth]{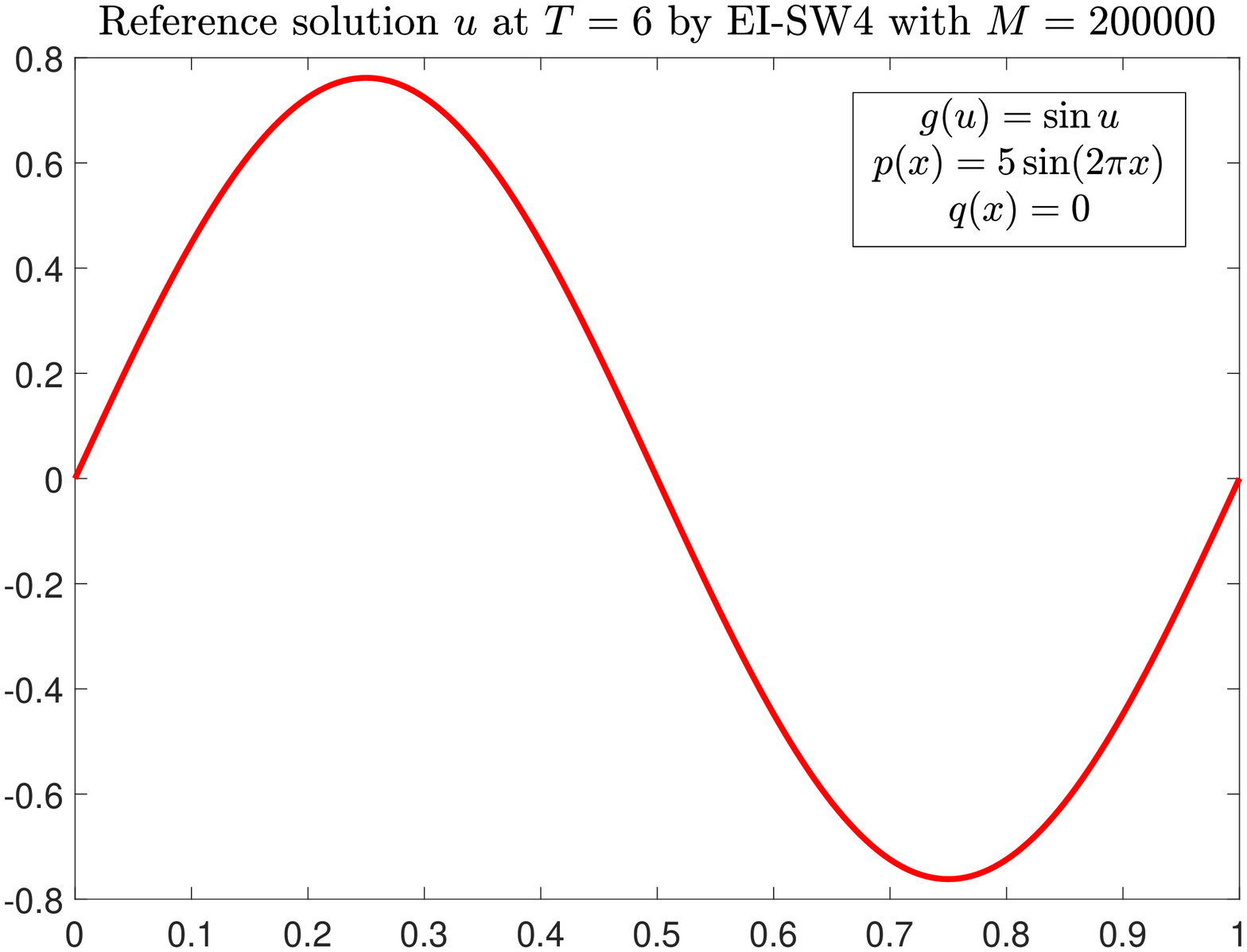}
         \caption{Reference solution $u$ at $T = 6$.}
         \label{fig-ex1-refsol-u}
\end{subfigure}
\begin{subfigure}[]{0.495\textwidth}
         \centering
         \includegraphics[width=\textwidth]{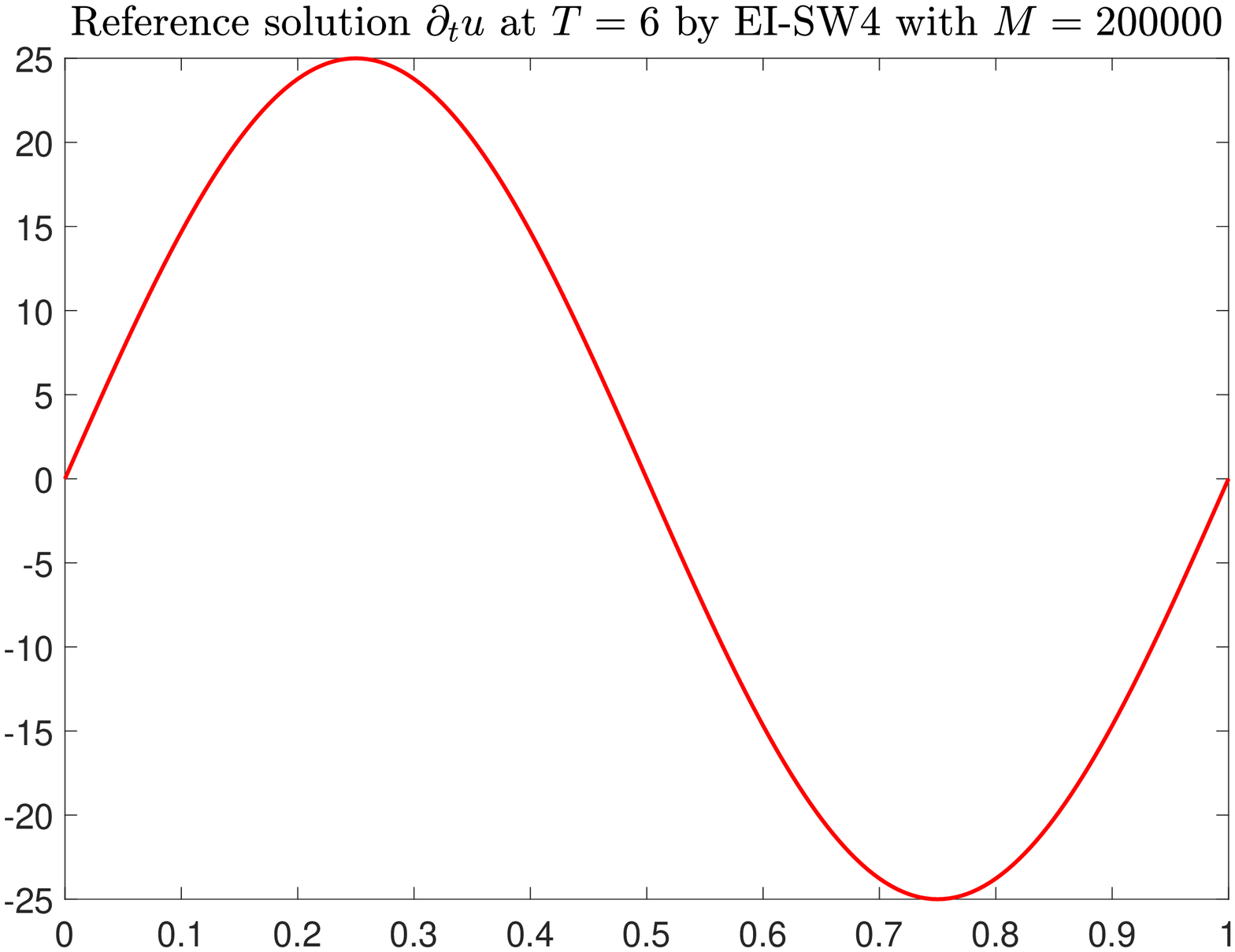}
         \caption{Reference solution $\partial_t u$ at $T = 6$.}
         \label{fig-ex1-refsol-ut}
\end{subfigure}  \\
\begin{subfigure}[]{0.75\textwidth}
         \centering
         \includegraphics[width=\textwidth]{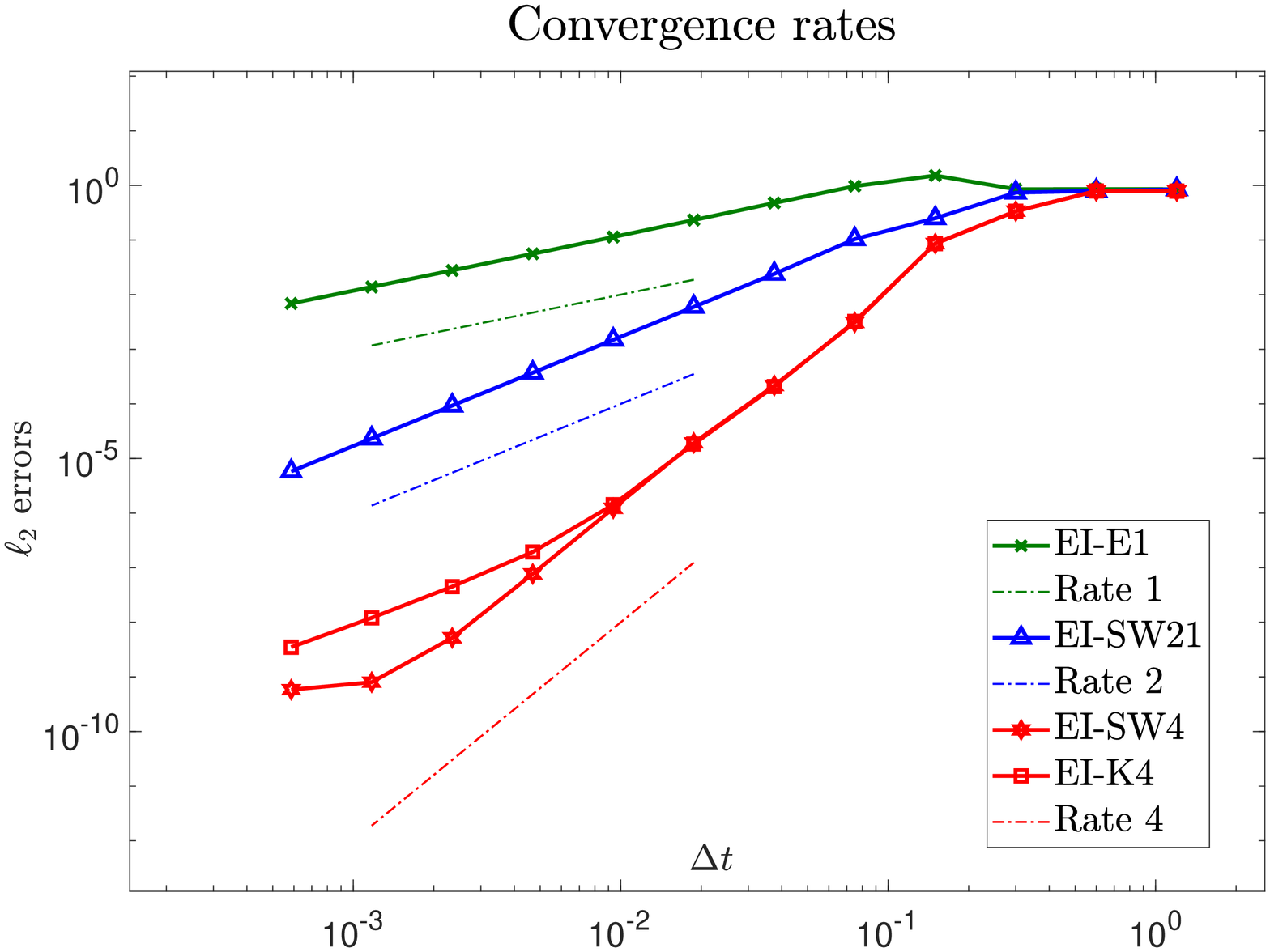}
         \caption{Convergence rates.}
         \label{fig-ex1-rate}
\end{subfigure}
\caption{Example \ref{ex-wave1}}
\end{figure}
\end{ex}

\begin{ex} \label{ex-wave2}
Consider equation \eqref{eq-example-wave} with $\alpha = 100,~\beta = 10^{-2},~\delta = 0,~\gamma = 10^{-3}$. The nonlinear source term is $g(u) = u|u|$. The initial conditions are 
\begin{align*}
p(x) = \begin{cases} 
2x \qquad &\text{if} \quad x \le \frac{1}{2}, \\
-2x+2 \quad &\text{if} \quad x > \frac{1}{2}, 
\end{cases} \quad \quad q (x) = \pi^2 \sin(\pi x). 
\end{align*}
We use {five} different schemes, namely {\ttfamily{EI-E1}}, {\ttfamily{EI-SW21}}, {\ttfamily{EI-SW22}} (both schemes with {$c_2 = 0.2$}), {\ttfamily{EI-SW4}}, and {\ttfamily{EI-K4}} to compute the solution at time $T = 15$ with $M \in \{20, 40, 80, \cdots, 20 \cdot 2^{12}\} $ time steps. The reference solution plotted in Figures \ref{fig-ex2-refsol-u} and \ref{fig-ex2-refsol-ut} is computed by using {\ttfamily{EI-K4}} with $M = 200000$. The errors are plotted in Figure \ref{fig-ex2-rate}. 
\begin{figure}[h]
\centering 
\begin{subfigure}[]{0.495\textwidth}
         \centering
         \includegraphics[width=\textwidth]{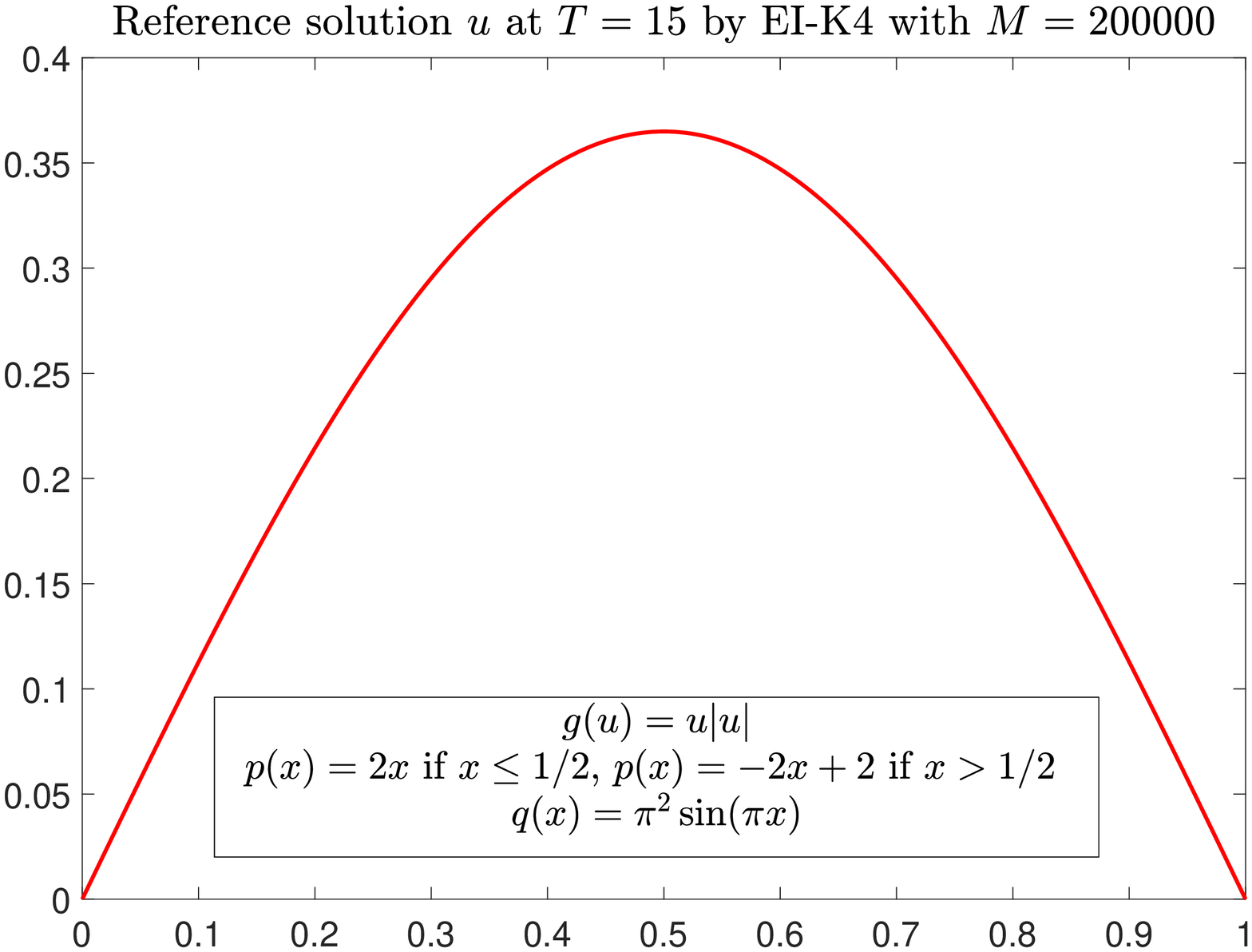}
         \caption{Reference solution $u$ at $T = 15$.}
         \label{fig-ex2-refsol-u}
\end{subfigure}
\begin{subfigure}[]{0.495\textwidth}
         \centering
         \includegraphics[width=\textwidth]{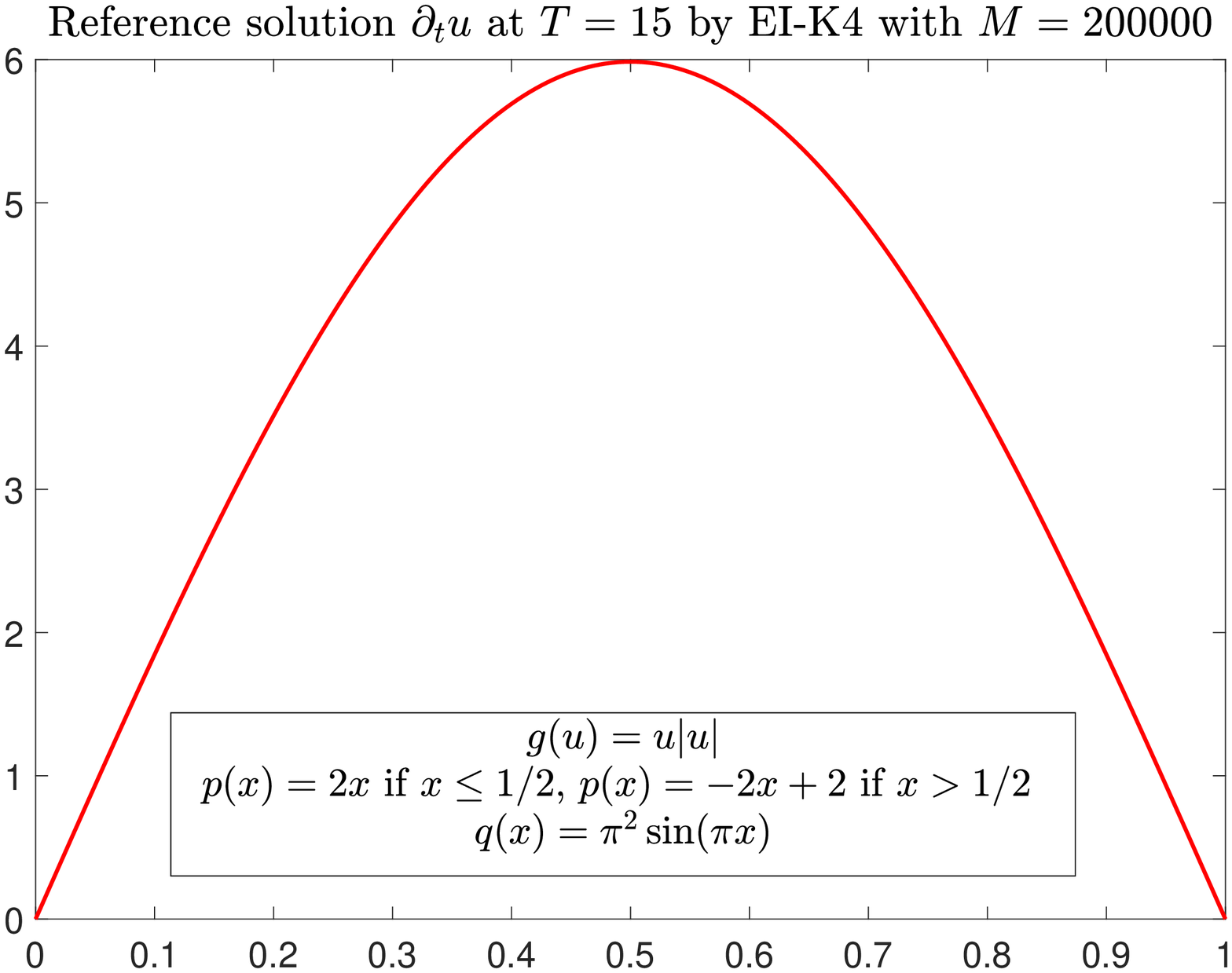}
         \caption{Reference solution $\pdt u$ at $T = 15$.}
         \label{fig-ex2-refsol-ut}
\end{subfigure}
 \\
\begin{subfigure}[]{0.8\textwidth}
         \centering
         \includegraphics[width=\textwidth]{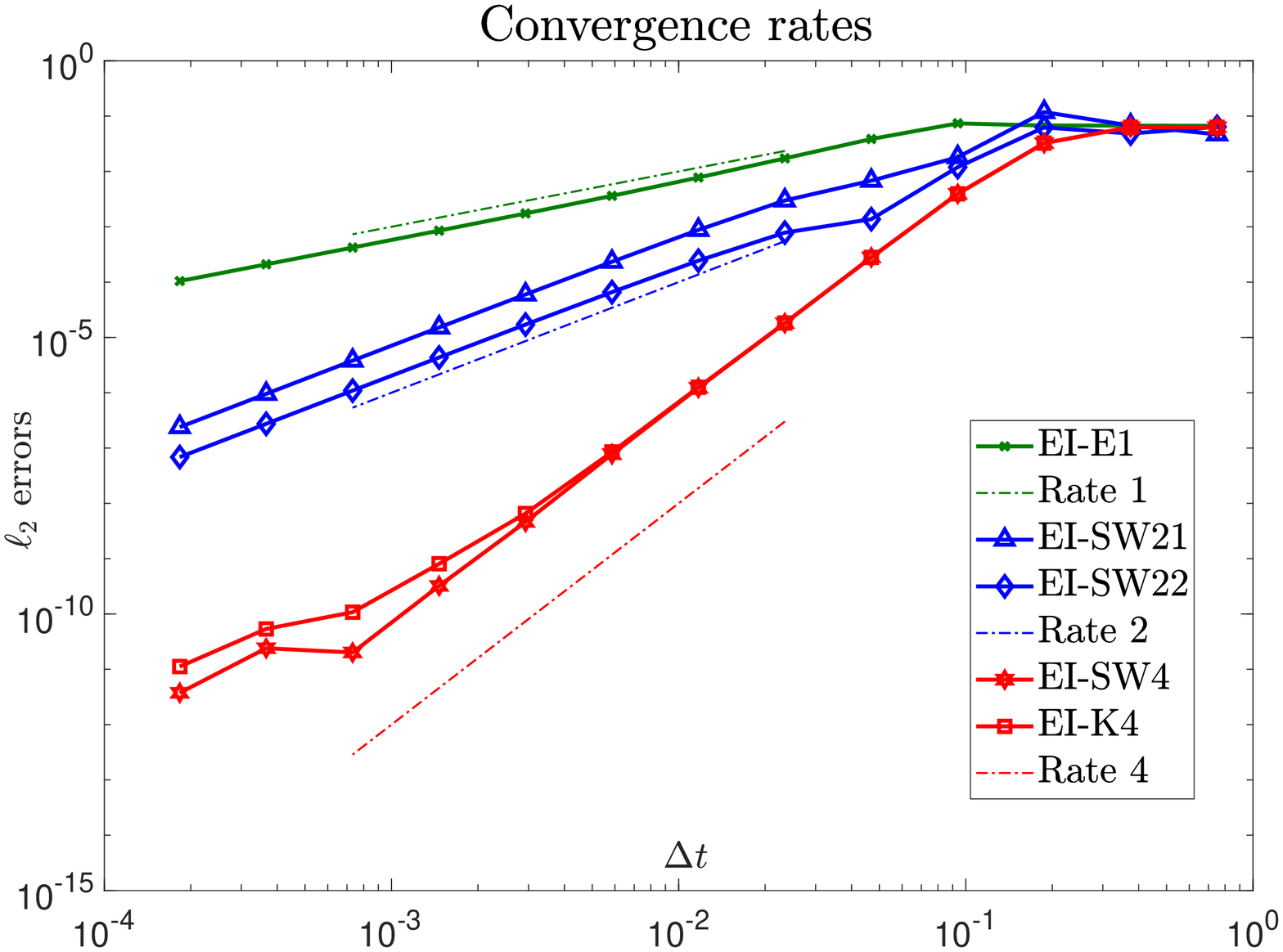}
         \caption{Convergence rates.}
         \label{fig-ex2-rate}
\end{subfigure}
\caption{Example \ref{ex-wave2}}
\end{figure}
\end{ex}

\begin{ex} \label{ex-wave3}
Consider equation \eqref{eq-example-wave} with $\alpha = 15,~\beta = 10^{-3},~  \delta = 1,~\gamma = 10^{-6}$. The nonlinear source term is $g(u) = u^3$. The initial conditions are $p(x) = 10 \sin(3\pi x),~q (x) = -10 \cos(3\pi x)  $. We use all five schemes, namely {\ttfamily{EI-E1}}, {\ttfamily{EI-SW22}} (with $c_2 = 0.9$), {\ttfamily{EI-K4}}, and {\ttfamily{EI-SW4}} to compute the solution at time $T = 30$ with $M \in \{20, 40, 80, \cdots, 20 \cdot 2^{11}\}$ time steps. We plot the reference solution computed with \textsc{EI-SW4}  and $M = 300000$ in Figures \ref{fig-ex3-refsol-u} and \ref{fig-ex3-refsol-ut}. Again the expected convergence rates are observed for each scheme and plotted in Figure \ref{fig-ex3-rate}. An order reduction occurs for {\ttfamily{EI-K4}} (reduction to order 2) while {\ttfamily{EI-SW4}} still preserves its convergence rate.

\begin{figure}[h]
\centering 
\begin{subfigure}[]{0.495\textwidth}
         \centering
         \includegraphics[width=\textwidth]{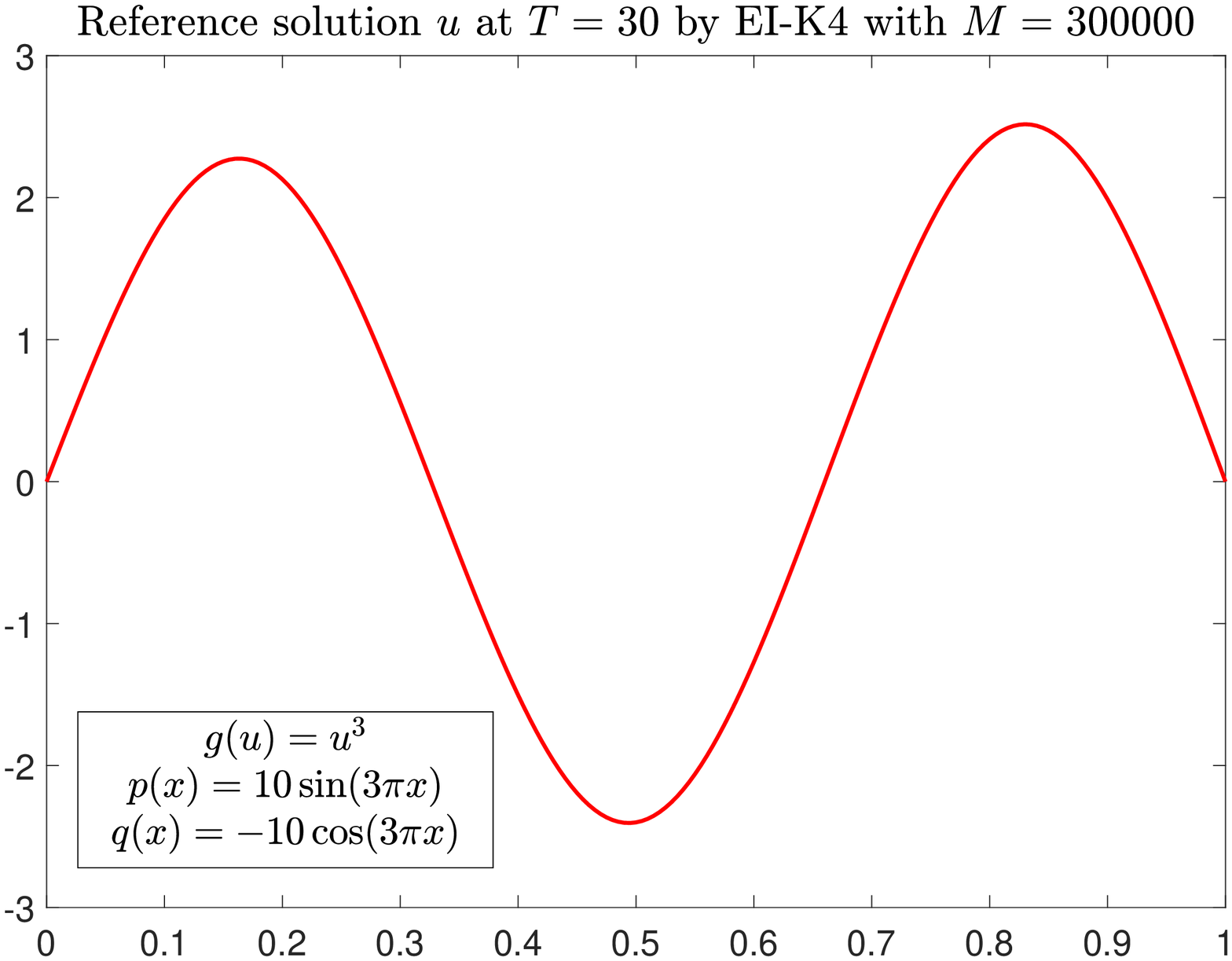}
         \caption{Reference solution $u$ at $T = 30$.}
         \label{fig-ex3-refsol-u}
\end{subfigure}
\begin{subfigure}[]{0.495\textwidth}
         \centering
         \includegraphics[width=\textwidth]{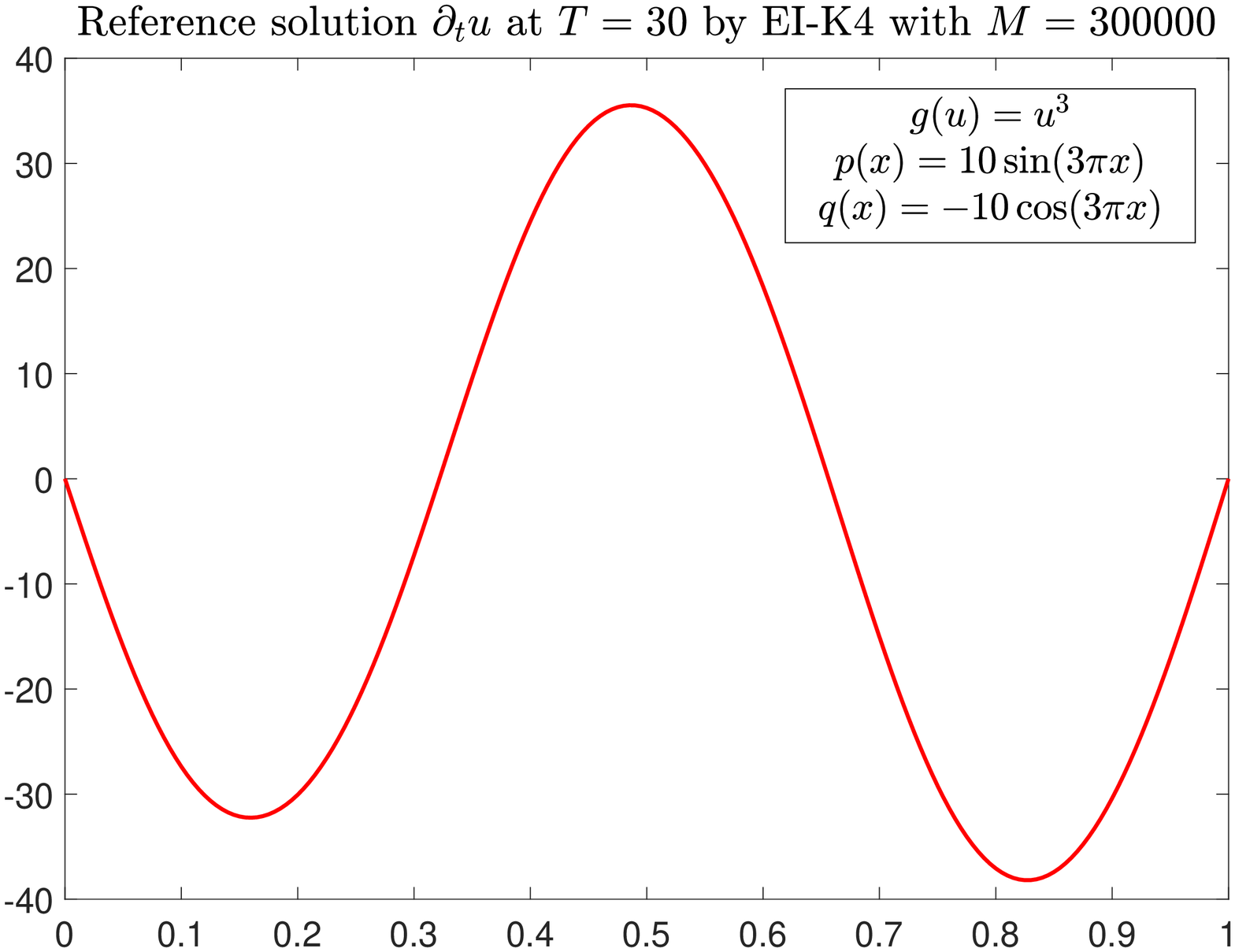}
         \caption{Reference solution $\pdt u$ at $T = 30$.}
         \label{fig-ex3-refsol-ut}
\end{subfigure}
 \\
\begin{subfigure}[]{0.8\textwidth}
         \centering
         \includegraphics[width=\textwidth]{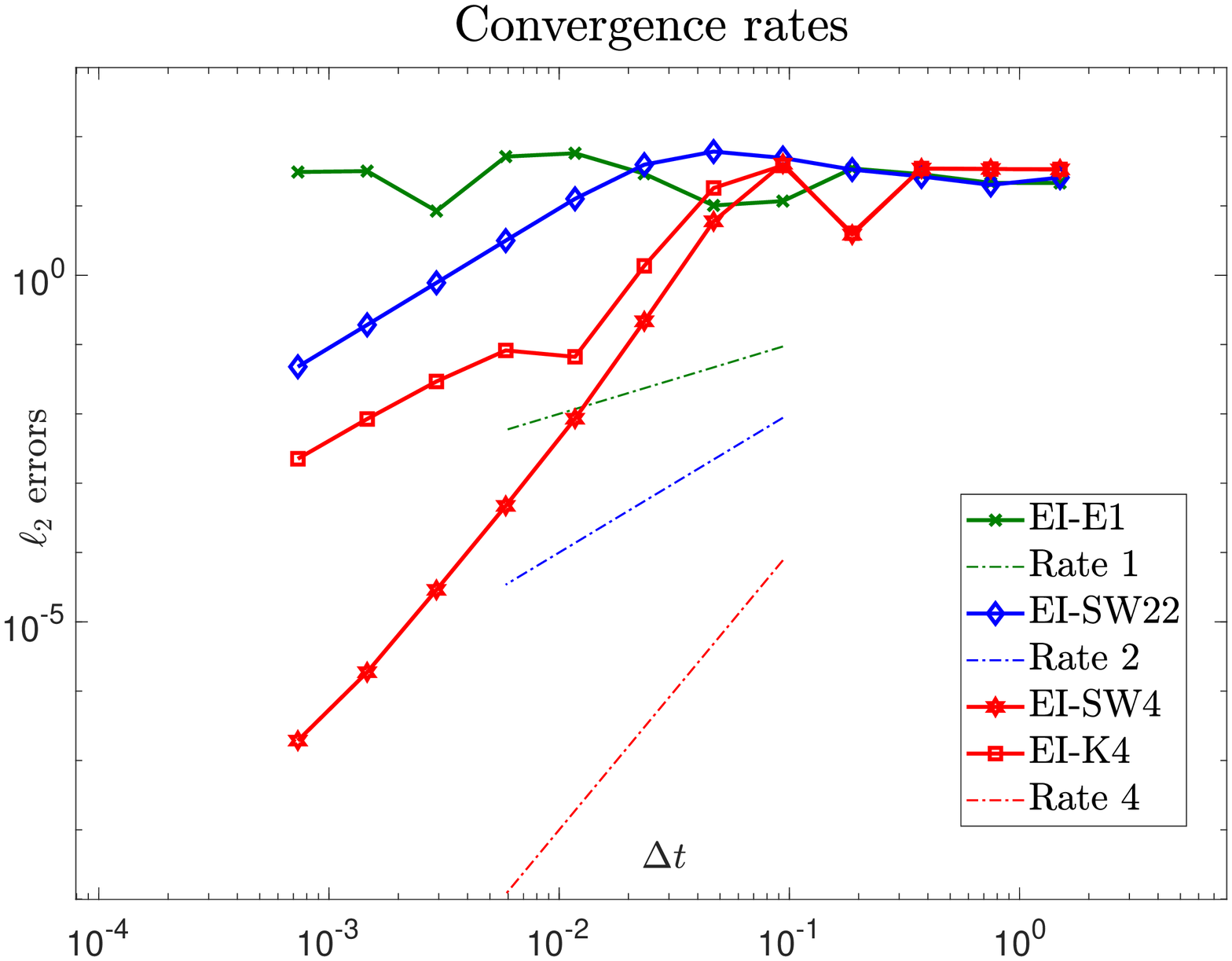}
         \caption{Convergence rates.}
         \label{fig-ex3-rate}
\end{subfigure}
\caption{Example \ref{ex-wave3}}
\end{figure}
\end{ex}

\begin{ex} \label{ex-wave4} This example concerns discontinuous initial conditions
\begin{align*} p(x) = 
\begin{cases} 
-1\qquad &\text{if} \quad x \le \frac{1}{2}, \\
5\qquad &\text{if} \quad x > \frac{1}{2}, 
\end{cases} \quad \quad q(x) = 0. 
\end{align*}
The other parameters are $\alpha = 5,~~\beta = 10^{-3},~\delta = 1,~\gamma = 10^{-4}$. The nonlinear term is $g(u) = |u|$. The approximate solutions at $T = 3$ are computed by using five exponential integrators, namely {\ttfamily{EI-E1}}, {\ttfamily{EI-SW21}}, {\ttfamily{EI-SW22}} (both schemes with $c_2 = 0.5$), {\ttfamily{EI-SW4}}, and {\ttfamily{EI-K4}} with $M \in \{20, 40, \dots, 20\cdot2^{13}\}$ time steps. The reference solution computed with $M = 300000$ time steps by using {\ttfamily{EI-SW4}} is plotted in Figures \ref{fig-ex4-refsol-u} and \ref{fig-ex4-refsol-ut}. The errors are plotted in Figure \ref{fig-ex4-rate}. We observe an order reduction to order 2 for the two fourth-order exponential integrators {\ttfamily{EI-SW4}} and {\ttfamily{EI-K4}} while the other integrators preserve their convergence rate. 
\begin{figure}[h]
\centering 
\begin{subfigure}[]{0.495\textwidth}
         \centering
         \includegraphics[width=\textwidth]{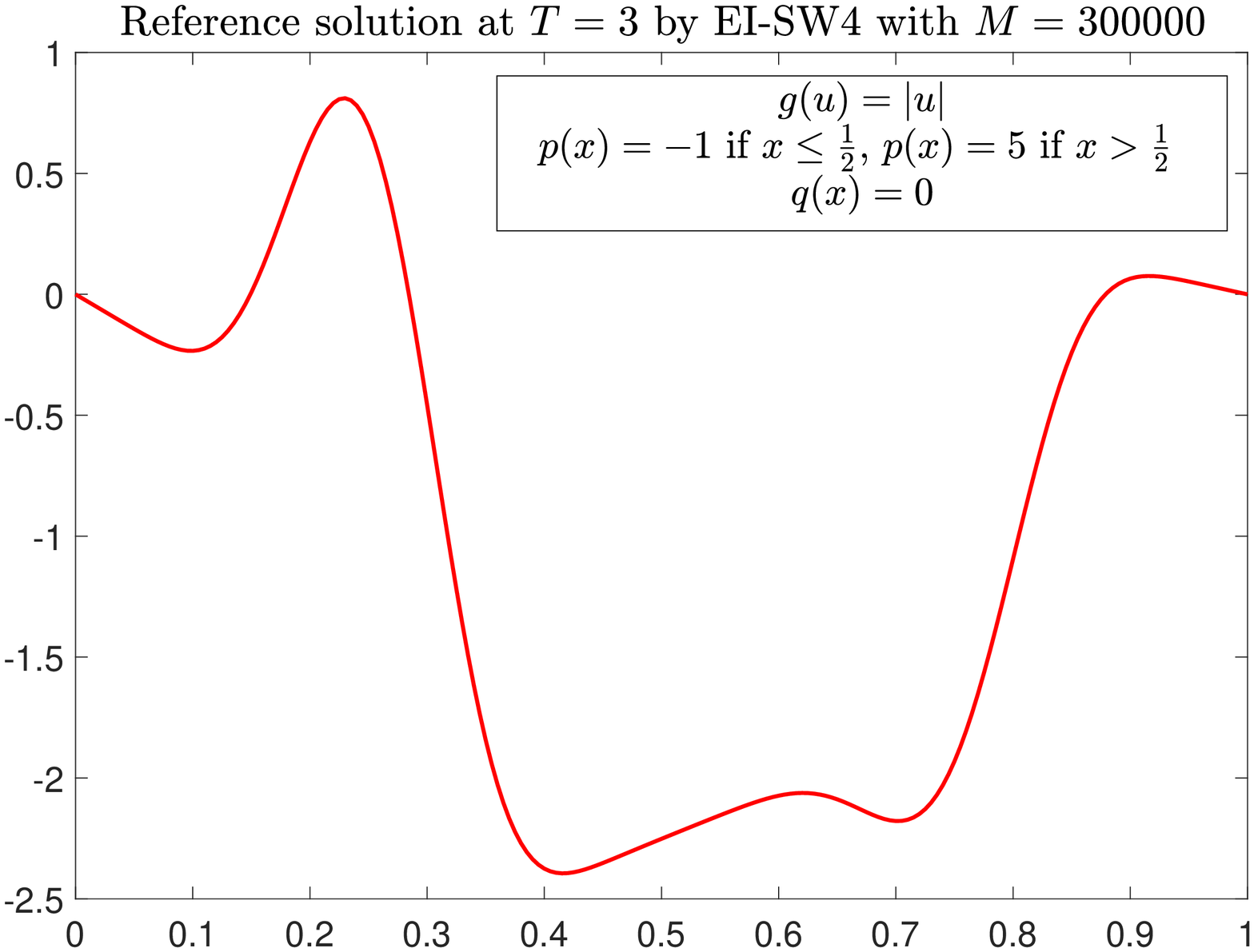}
         \caption{Reference solution $u$ at $T = 3$.}
         \label{fig-ex4-refsol-u}
\end{subfigure}
\begin{subfigure}[]{0.495\textwidth}
         \centering
         \includegraphics[width=\textwidth]{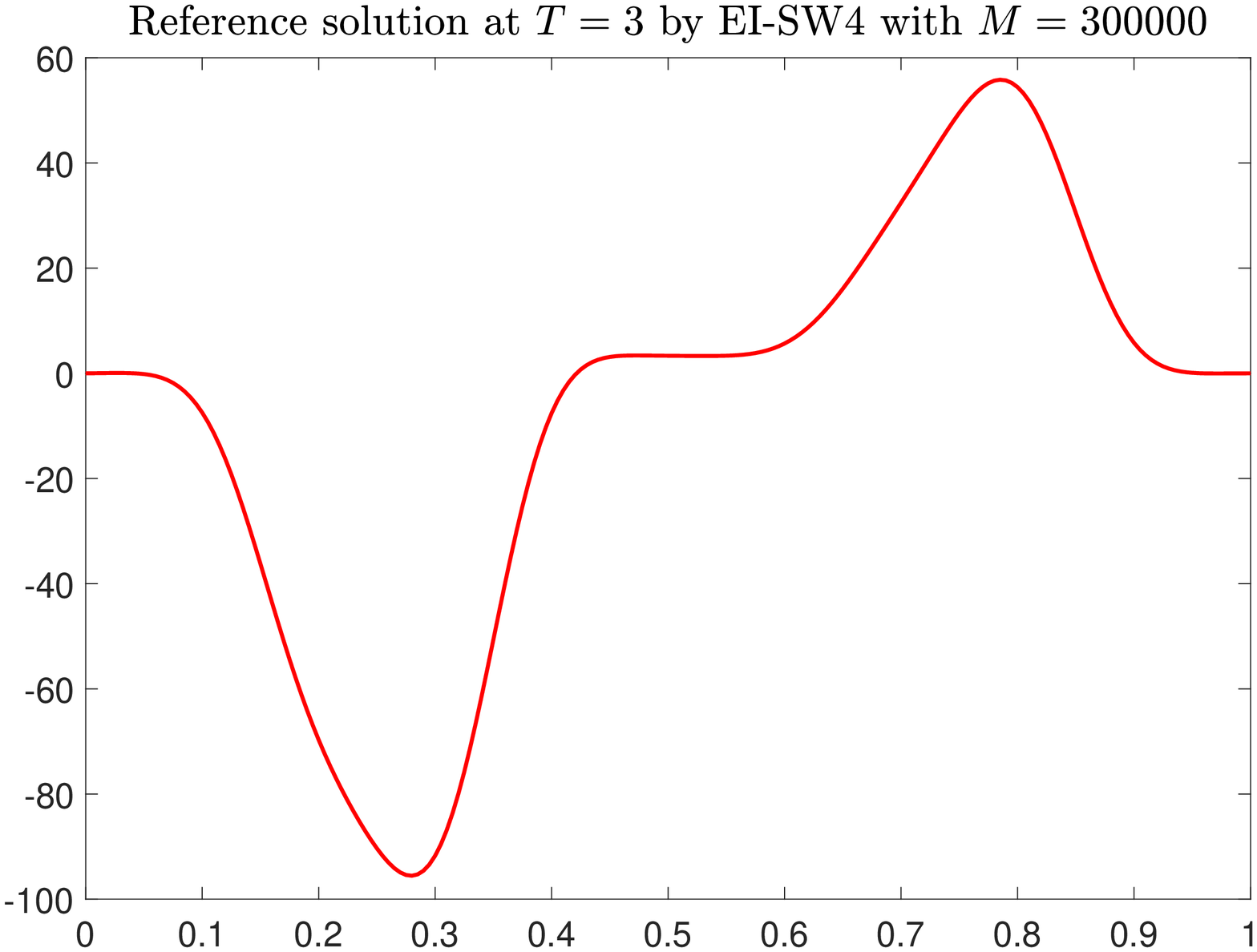}
         \caption{Reference solution $\pdt u$ at $T = 3$.}
         \label{fig-ex4-refsol-ut}
\end{subfigure}
 \\
\begin{subfigure}[]{0.8\textwidth}
         \centering
         \includegraphics[width=\textwidth]{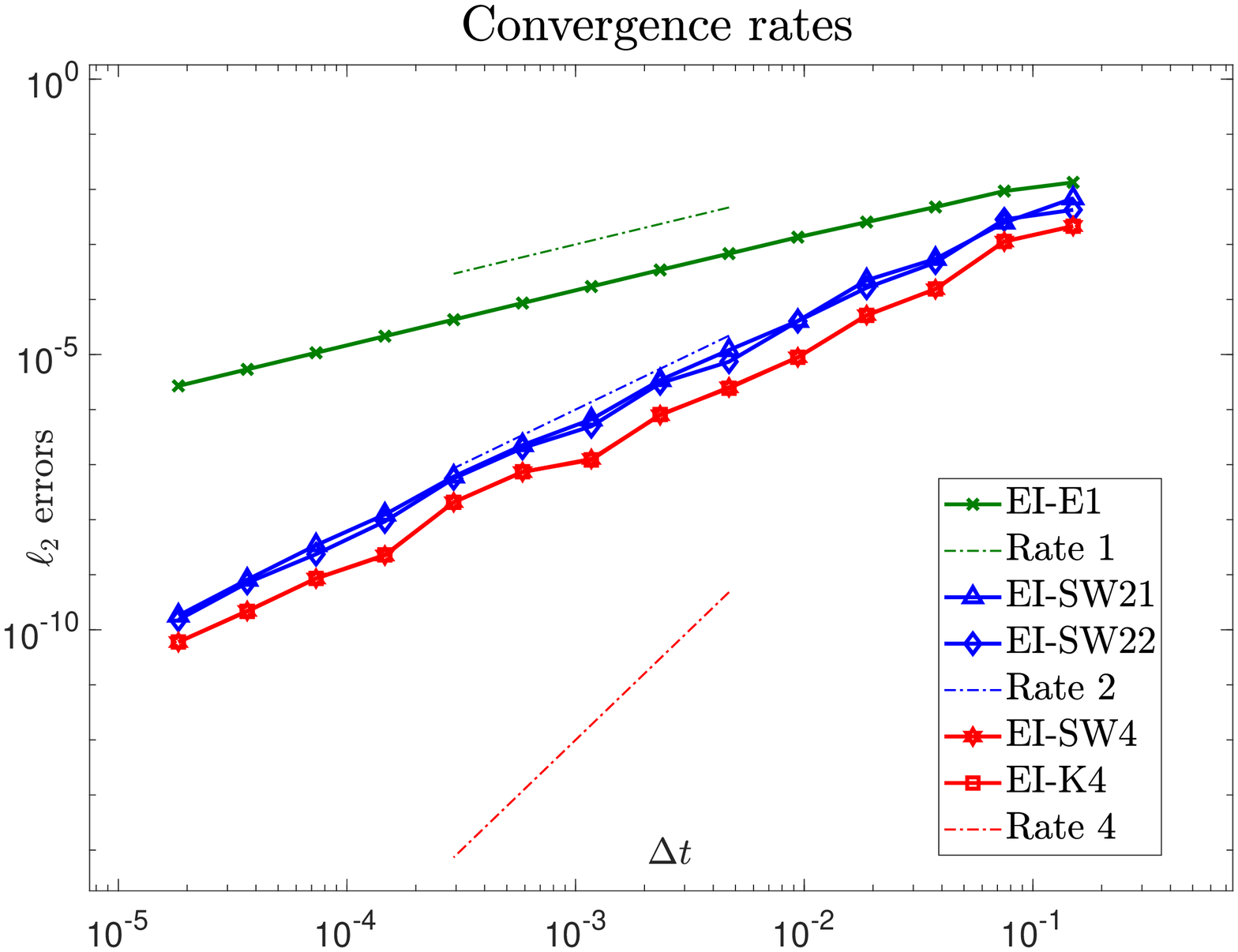}
         \caption{Convergence rates.}
         \label{fig-ex4-rate}
\end{subfigure}
\caption{Example \ref{ex-wave4}}
\end{figure}

\end{ex}

\begin{ex} \label{ex-wave5} The last example concerns two nonlinear terms, namely $g(u) = -u |u|^3$, $h(w) = -w|w|$. The other parameters are $\alpha = 50,~\beta = 10^{-6},~\delta = 10,~\gamma = 10^{-3}$. The initial conditions are $p(x) = 20 \sin(4 \pi x)$ and $q(x) = -25 \cos(3\pi x) $. The approximate solution at $T=1$ are computed by using five exponential integrators, namely {\ttfamily{EI-E1}}, {\ttfamily{EI-SW21}}, {\ttfamily{EI-SW22}} (both schemes with $c_2 = 0.85$), {\ttfamily{EI-SW4}}, and {\ttfamily{EI-K4}} with $M \in \{160, 320, \dots, 160 \cdot 2^{10}\}$ time steps. The reference solution is computed by {\ttfamily{EI-SW4}} with $M = 800000$ and plotted in Figures \ref{fig-ex5-refsol-u} and \ref{fig-ex5-refsol-ut}. The convergence rates are plotted in Figure \ref{fig-ex5-rate}. The two fourth-order exponential integrators  {\ttfamily{EI-SW4}} and {\ttfamily{EI-K4}} show order reductions. While {\ttfamily{EI-SW4}} works still well with an order reduction to order 3; {\ttfamily{EI-K4}} on the other hand works badly. 
\begin{figure}[h]
\centering 
\begin{subfigure}[]{0.495\textwidth}
         \centering
         \includegraphics[width=\textwidth]{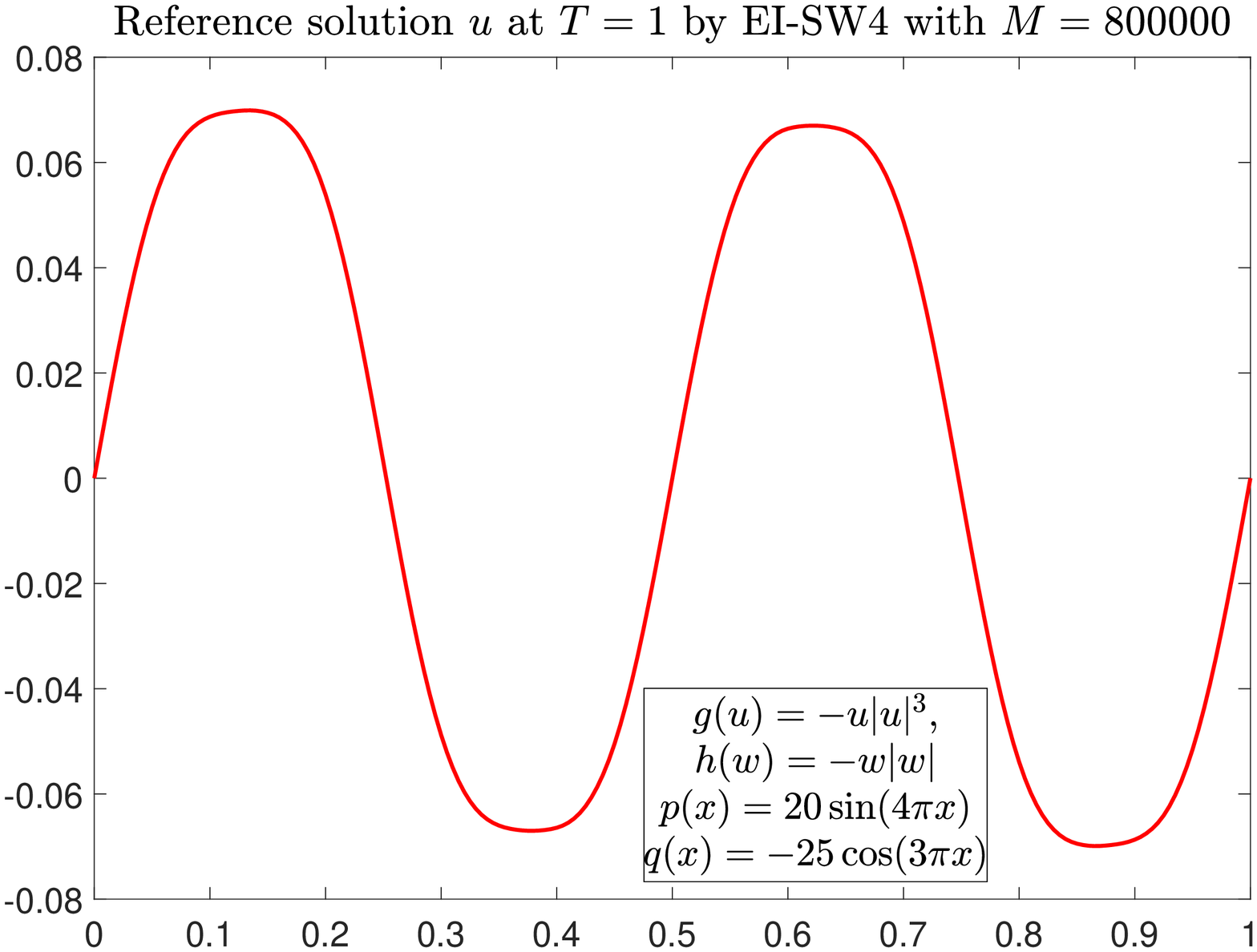}
         \caption{Reference solution $u$ at $T = 1$.}
         \label{fig-ex5-refsol-u}
\end{subfigure}
\begin{subfigure}[]{0.495\textwidth}
         \centering
         \includegraphics[width=\textwidth]{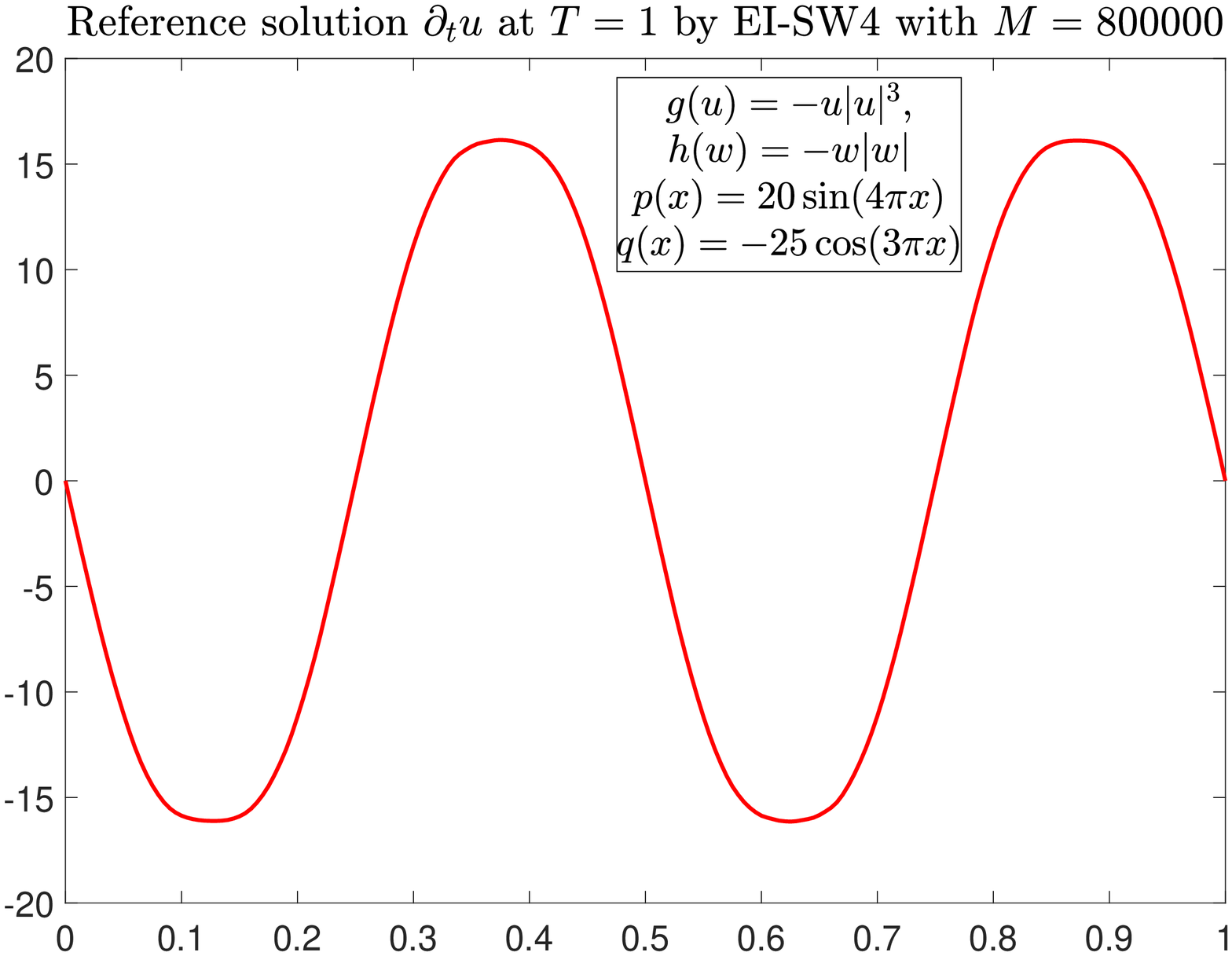}
         \caption{Reference solution $\pdt u$ at $T = 1$.}
         \label{fig-ex5-refsol-ut}
\end{subfigure}
 \\
\begin{subfigure}[]{0.8\textwidth}
         \centering
         \includegraphics[width=\textwidth]{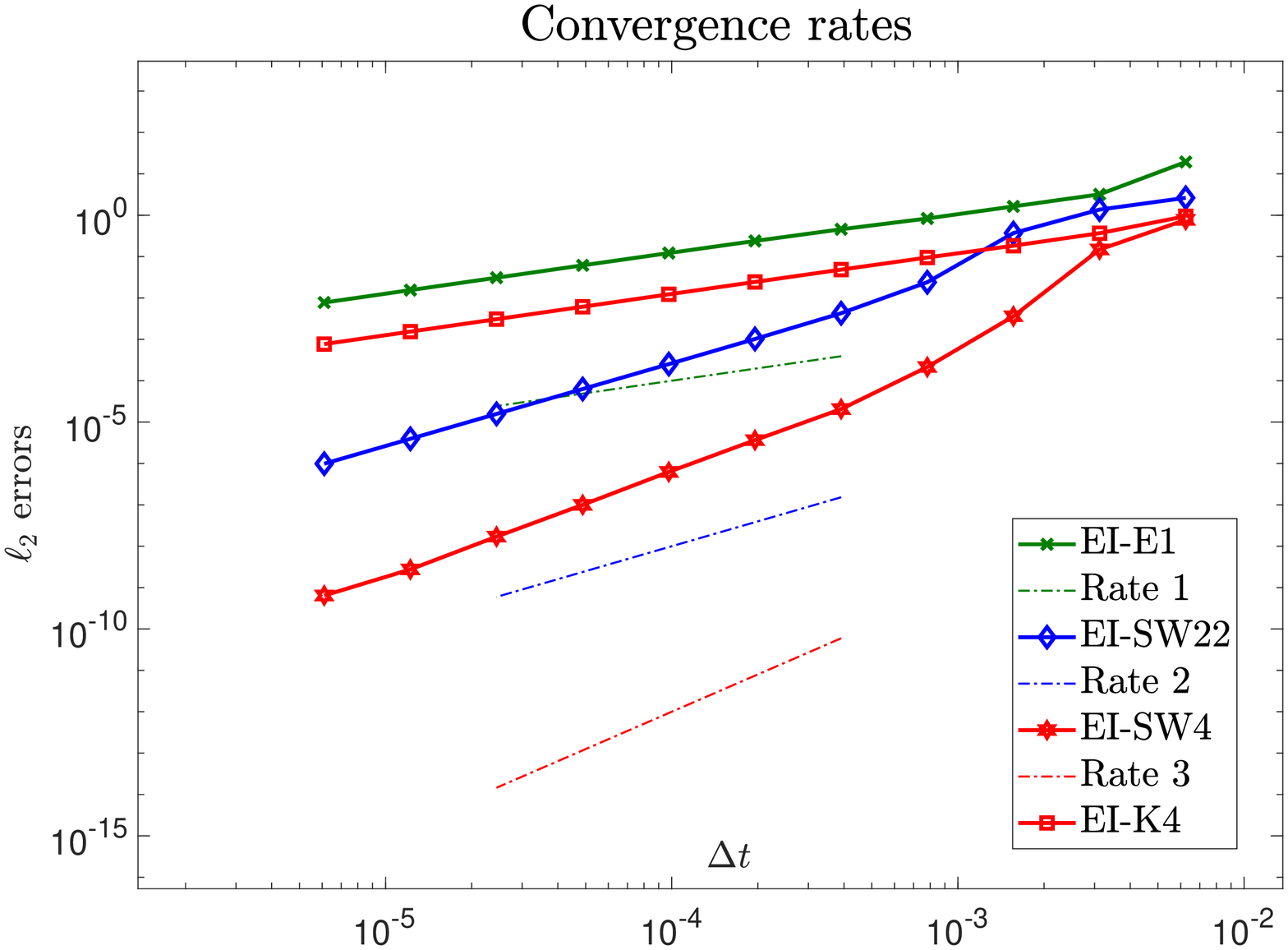}
         \caption{Convergence rates.}
         \label{fig-ex5-rate}
\end{subfigure}
\caption{Example \ref{ex-wave5}}
\end{figure}
\end{ex}

\subsection{Railway track model}
Assume that a track beam is made of Kelvin-Voigt material. The resulting railway track model is a semilinear PDE on $\Omega = (0,L)$: 
\begin{subequations}\label{eq-example-beam}
\begin{align} 
\pdtt u &+ \pdxx (\alpha \pdxx u + \beta \pdxx(\pdt u) ) + \gamma \pdt u  + \delta u = -lu^3, \quad x \in \Omega, t \in (0,T], \\
u(t,0) &= u(t,L) = 0, \\
\alpha \pdxx u (t,0) &+ \beta \pdxx(\pdt u) (t,0) =     \alpha \pdxx u (t,L) + \beta \pdxx(\pdt u) (t,L) = 0, \\
u(0,x) &= p(x), \quad \pdt u (0,x) = q(x). 
\end{align}
\end{subequations}
Denote the closed self-adjoint positive operator $\cA_0$ on $L^2(0,L)$ as 
\begin{align*}
\cA_0 \phi &\coloneqq \pdxxxx \phi, \\
D(\cA_0) &\coloneqq \{\phi \in H^4(\Omega) \mid \phi(0) = \phi(L) = 0,~~ \phi''(0) = \phi''(L) = 0 \}.
\end{align*}
Concerning the analysis of the linear operators, we refer to the literature \cite{Mor20,BanIto97}.
We use finite differences to discretize the operator $\cA_0$ with an equidistant space mesh $x_i = i \Delta x,~~ i \in \{0,\dots,N+1\} $, where $N$ is a given integer and $\Delta x = \frac{L}{N + 1}$. Then the discrete operator $\cA_0$ is given by the matrix $S_b \in \R^{N \times N}$, defined as below
\begin{align} \label{eq-discrete-Sb}
S_b = \frac{1}{\Delta x^4} 
\bmat{
5 & -4 & 1 & 0 & 0 & 0  & \dots & 0 & 0 & 0  & 0 & 0  \\
-4 & 6 & -4 & 1 & 0 & 0 & \dots & 0 & 0 & 0  & 0 & 0 \\
1 & -4 & 6 & -4 & 1 & 0 & \dots & 0 & 0 & 0  & 0 & 0  \\
0 & 1 & -4 & 6 & -4 & 1  & \dots & 0 & 0 & 0 & 0  & 0  \\
\vdots & \ddots & \ddots  & \ddots  & \ddots  & \ddots   & \ddots  & \ddots  & \ddots  & \ddots  & \ddots  & \vdots \\
\vdots & \ddots & \ddots  & \ddots  & \ddots  & \ddots   & \ddots  & \ddots  & \ddots  & \ddots  & \ddots  & \vdots \\
\vdots & \ddots & \ddots  & \ddots  & \ddots  & \ddots   & \ddots  & \ddots  & \ddots  & \ddots  & \ddots  & \vdots \\
\vdots & \ddots & \ddots  & \ddots  & \ddots  & \ddots   & \ddots  & \ddots  & \ddots  & \ddots  & \ddots  & \vdots \\
\vdots & \ddots & \ddots  & \ddots  & \ddots  & \ddots   & \ddots  & \ddots  & \ddots  & \ddots  & \ddots  & \vdots \\
0 & \dots & \dots & \dots & \dots &  \dots   & \ddots & 1 & -4 & 6 & -4 & 1 \\
0& \dots & \dots & \dots & \dots &  \dots   & \ddots & 0 & 1 & -4 & 6 & -4 \\
0& \dots & \dots & \dots & \dots &  \dots   & \dots & \dots & 0 & 1 & -4 & 5 \\
}.
\end{align}

\begin{ex} \label{ex-beam1} Consider equation \eqref{eq-example-beam} with $\alpha = 15,~\beta = 3 \cdot 10^{-6},~\delta = 10,~\gamma = 3\cdot 10^{-4}$. The nonlinear term is $g(u) = -5u^3$. The initial conditions are 
\begin{align*}
p(x) = 5 \me^{-100\left(x-\frac{2}{3}\right)^2},\quad q(x) = 0. 
\end{align*}
For our numerical solution, the space interval $\Omega = (0,1)$ is divided into 300 equidistant subintervals. We compute approximate solutions at $T = 5$ with four exponential integrators {\ttfamily{EI-E1}}, {\ttfamily{EI-SW22}} ($c_2 = 0.9$), {\ttfamily{EI-SW4}}, and  {\ttfamily{EI-K4}} with $M \in \{160,320,\dots, 160 \cdot 2^{10}\}$ time steps. We compare these numerical results with the reference solution evaluated by {\ttfamily{EI-K4}} with $M = 600000$ time steps. The reference solution is plotted in Figures \eqref{fig-ex1b-refsol-u} and \eqref{fig-ex1b-refsol-ut}. Notice that the magnitude of the velocity $\pdt u$ is extremely big. The errors are plotted in Figure \ref{fig-ex1b-rate}. The four exponential integrators preserve their convergence rate. Since the matrix $S_b$ is stiffer than the one $S_w$, the computation for beam equations is more expensive than for wave equations. However, we note that solving beam equations with exponential integrators is a good option. Some comparisons in the next section will elucidate this point.

\begin{figure}[h]
\centering 
\begin{subfigure}[]{0.495\textwidth}
         \centering
         \includegraphics[width=\textwidth]{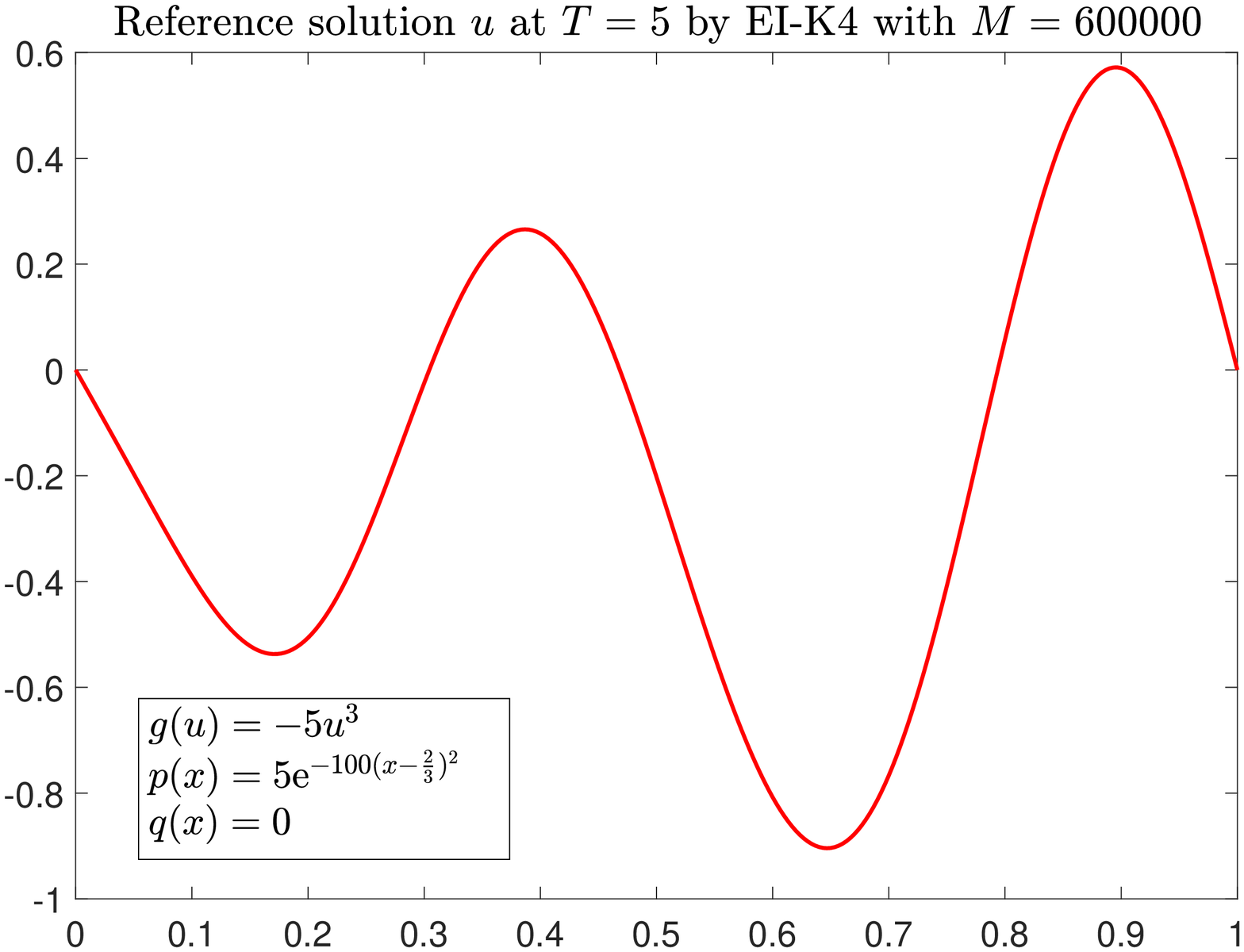}
         \caption{Reference solution $u$ at $T = 5$.}
         \label{fig-ex1b-refsol-u}
\end{subfigure}
\begin{subfigure}[]{0.495\textwidth}
         \centering
         \includegraphics[width=\textwidth]{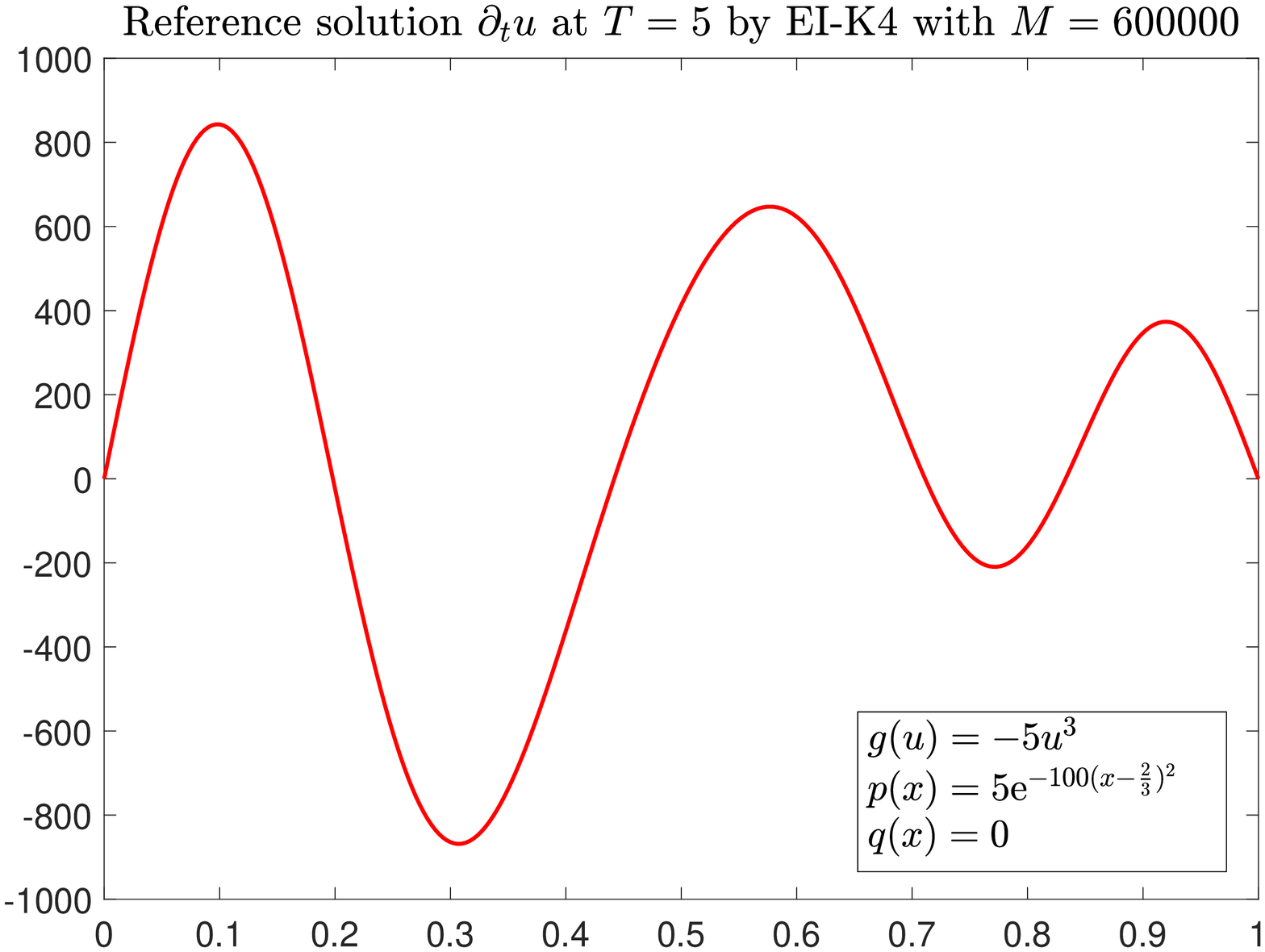}
         \caption{Reference solution $\pdt u$ at $T = 5$.}
         \label{fig-ex1b-refsol-ut}
\end{subfigure}
 \\
\begin{subfigure}[]{0.8\textwidth}
         \centering
         \includegraphics[width=\textwidth]{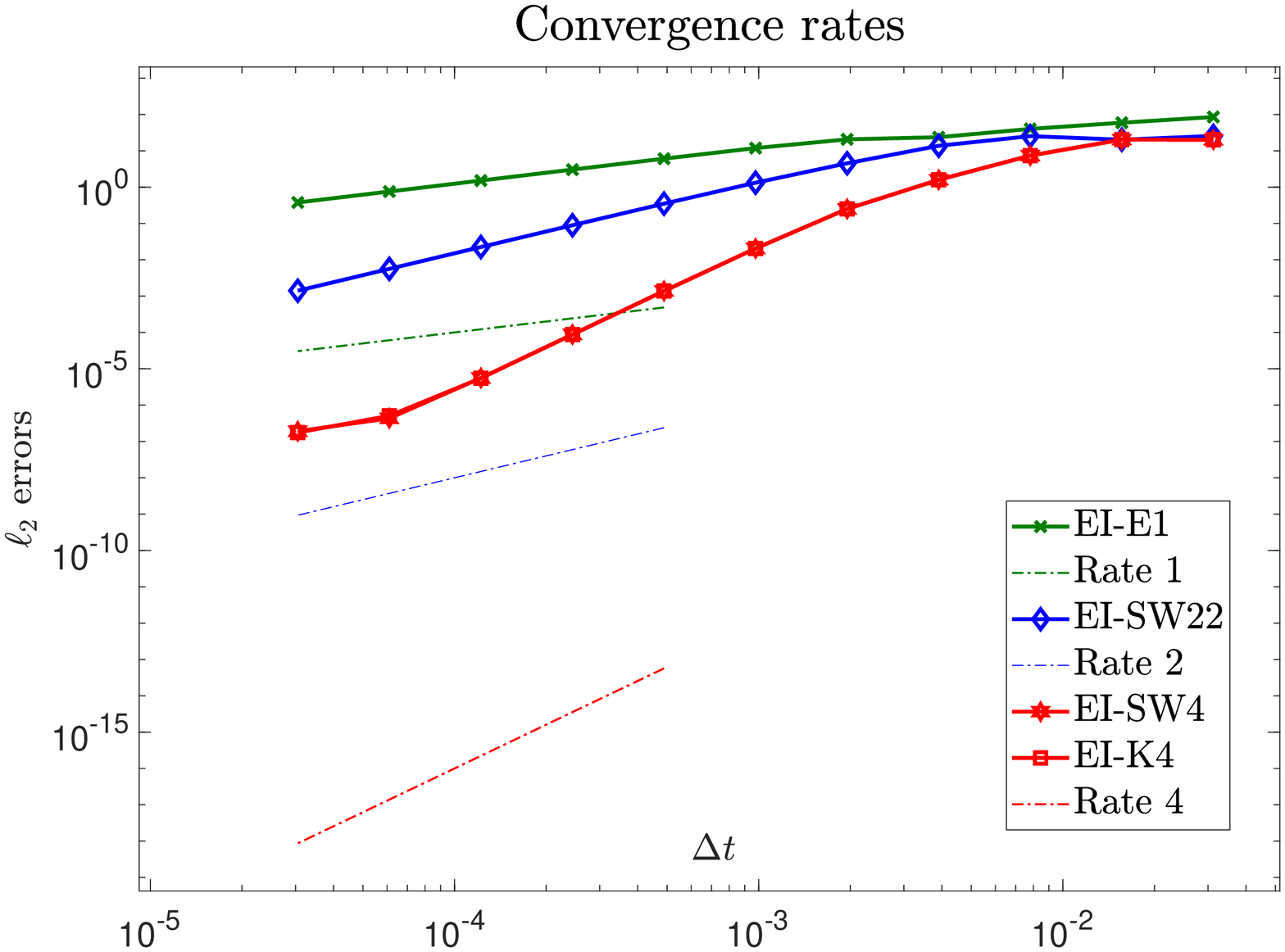}
         \caption{Convergence rates.}
         \label{fig-ex1b-rate}
\end{subfigure}
\caption{Example \ref{ex-beam1}}
\end{figure}
\end{ex}

\subsection{Comparisons with standard integrators} \label{sec-comparison}
Some comparisons between our approach and standard integrators will be presented to clarify the efficiency. In particular, we consider the explicit Runge-Kutta method {\ttfamily{ode45}} from MATLAB. Note that {\ttfamily{ode45}} needs sufficiently small time steps to guarantee stability. The CFL condition depends on the set up of our model. For example, the larger the parameter $\alpha$ we choose in the example, the smaller the time step $\Delta t$ has to been chosen. Besides, the maximum step size depends on the type of equation, i.e. the beam equation is stiffer than the wave equation. In particular, the relation between $\Delta t$ and $\Delta x$ is of the form: $\Delta t \sim (\Delta x)^2$ for beam equations, and $\Delta t \sim \Delta x$ for wave equations.  

Stiff problems are often solved with implicit schemes. Therefore, we will make another comparison with a class of implicit Runge-Kutta methods, namely the Radau IIA methods  (see \cite[Section IV-5]{HaiWan96}). Though these methods do not require any CFL condition to guarantee stability, their computational cost is high since they require the solution of linear systems with large matrices. Below, we illustrate by some examples that both explicit and implicit Runge-Kutta methods are more expensive than our exponential integrators in the present context. 

\begin{ex} \textit{Exact matrix exponential versus ode45 and Radau scheme for a linear example.} \label{ex-compare-linsys}
Consider a linearized version of equation \eqref{eq-example-wave} with $\alpha =100,~\beta = 10^{-2},~\delta = 10^{-2} ,~\gamma = 10^{-6}$. The initial conditions are $p(x) = 5 \sin(2\pi x)$ and $q(x) = 0$. We consider the space interval $\Omega = (0,1)$ with $N = 200$ (number of grid points). 

We compute the solution at time $T = 10$ by using different methods. For any exponential integrator, the solution is obtained immediately by the formula 
\begin{align*}
y(T) = \bmat{u(T) \\ w(T)} = \me^{TA} \bmat{p \\ q}, \quad A = \bmat{0 & I \\ -\alpha S_w - \beta I & - \beta S_w - \gamma I},
\end{align*}
where $S_w$ was defined in \eqref{eq-discrete-Sw}. The computational time for using this approach is 0.036s. 

For comparisons, we use an explicit Runge-Kutta method, namely {\ttfamily{ode45}} from MATLAB to obtain the solution at final time $T =10$ with various tolerances. Due to the stability reason, {\ttfamily{ode45}} needs a huge number of time steps. The minimum number of time steps $M$ which  {\ttfamily{ode45}} needs to achieve the corresponding accuracy are presented in the second column of Table~\ref{tab-ode45-linear-wave}. The relating computational time is reported in the third column.
As implicit method, we use the Radau IIA scheme. The minimum number of time steps which {\ttfamily{Radau}}\footnote{ \url{http://www.unige.ch/~hairer/prog/MatlabStiff.7z}} 
needs to obtain the corresponding accuracy is shown in the fourth column of Table \ref{tab-ode45-linear-wave}. The relating computational time is reported in the fifth column. The implicit scheme is more efficient than the explicit one for solving this linear equation.
Obviously to tackle the linear case, both {\ttfamily{ode45}} and {\ttfamily{Radau}} are expensive choices. 

\begin{table}[h!]   \centering
\begin{tabular}{ |P{1.5cm} || P{2cm} |P{2cm} || P{2cm} | P{2cm}|}
\hline
\multirow{2}{*}{Tolerance}   &  \multicolumn{2}{c||}{\ttfamily{ode45}} &  \multicolumn{2}{c|}{{\ttfamily{Radau}}}  \\  \cline{2-5} 
&  $M$ & Time & $M$  & Time\\
\hline 
$10^{-3}$ & 55621    & 18.17s & 32 & 1.33s\\ 
$10^{-6}$ & 55865& 17.63s & 58 & 2.09s  \\ 
$10^{-10}$& 888970 & 26.76s & 99 & 4.03s\\ 
$10^{-12}$&224733 &  68.30s & 284 & 5.25s \\ 
\hline 
\end{tabular} \caption{Number of time steps $M$ that {\ttfamily{ode45}} and {\ttfamily{Radau}} need to compute the solution at $T = 10$ and the corresponding computational time. On the other hand, an exponential integrator needs just 0.036s to get the exact solution.} \label{tab-ode45-linear-wave}
\end{table}

\end{ex}

\begin{ex} \textit{{\ttfamily{EI-K4}} verus ode45 and Radau for a semilinear wave equation.} \label{ex-compare-semisys-wave}
Consider the equation \eqref{eq-example-beam} with $\alpha = 100,~\beta = \gamma = 10^{-3},~\delta = 10$. The nonlinear source term is $g(u) = u^2$. The initial conditions are 
\begin{align*}
p(x) = \begin{cases} 
2x \qquad &\text{if} \quad x \le \frac{1}{2}, \\
-2x+2 \quad &\text{if} \quad x > \frac{1}{2}, 
\end{cases} \quad \quad q (x) = \pi^2 \sin(\pi x). 
\end{align*}
We compute the solution at the final time $T = 15$ by {\ttfamily{EI-K4}}, {\ttfamily{ode45}}, and {\ttfamily{Radau}}. To solve semilinear equations, three integrators need a sufficient number of substeps to attain the solution at final time. Corresponding to the desired accuracy, the number of steps as well as the required computational time of two schemes are reported in the Table \ref{tab-ode45-semilinear-wave}. For all accuracies, {\ttfamily{EI-K4}} is more efficient than the two other schemes. Even though {\ttfamily{Radau}} needs less number of time steps to achieve the desired accuracy, the computational cost is really high since it needs to solve a linear system involving a large matrix at each time step. 
\begin{table}[h!]   \centering
\begin{tabular}{ |P{1.5cm} || P{1.25cm} |P{1.25cm} || P{1.25cm} | P{1.25cm}|| P{1.25cm} |P{1.25cm} |}
\hline
\multirow{2}{*}{Tolerance}   &  \multicolumn{2}{c||}{\ttfamily{EI-K4}} &  \multicolumn{2}{c||}{\ttfamily{ode45}} &
\multicolumn{2}{c|}{\ttfamily{Radau}} \\  
\cline{2-7} 
&  $M$ & Time & $M$  & Time & $M$  & Time\\
\hline 
$10^{-4}$ & 20   & 0.36s & 123149 & 41.05s & 1170 & 213.01s\\ 
$10^{-6}$ & 640 & 6.56s & 128693 & 42.75s & 1734 & 319.57s  \\ 
$10^{-8}$& 2560 & 14.26s & 167861 & 54.85s & 2128 & 273.88s\\ 
$10^{-10}$& 20480 & 99.65s & 378641 & 119.29s & 3917 & 827.20s\\ 
\hline 
\end{tabular} \caption{Comparison among {\ttfamily{EI-K4}}, {\ttfamily{ode45}}, and {\ttfamily{Radau}}  for a semilinear wave equation. The number of time steps for the exponential integrator is empirically chosen to reach the prescribed accuracy.} \label{tab-ode45-semilinear-wave}
\end{table} 

\end{ex}

\begin{ex} \textit{{\ttfamily{EI-K4}} verus ode45 and Radau for a semilinear beam equation.}  \label{ex-compare-semisys-beam}
We repeat example \ref{ex-beam1} with a smaller matrix $S_b$, i.e., the space interval $\Omega = (0,1)$ is divided into 200 equidistant subintervals. The solution at time $T=1$ is computed with {\ttfamily{EI-K4}}, {\ttfamily{ode45}}, and {\ttfamily{Radau}}. The number of time steps and the corresponding computational time are presented in Table \ref{tab-ode45-semilinear-beam} for some desired tolerances. As we mentioned in the beginning of this section, this problem is stiffer than the wave equation. To tackle this challenging situation, {\ttfamily{ode45}} needs more than 1 millions time steps to achieve the solution even for a low tolerance like $10^{-2}$. On the other hand, the implicit {\ttfamily{Radau}} scheme requires less number of time steps to obtain a desired accuracy. However, its computational cost is very high.
\begin{table}[h!]   \centering
\begin{tabular}{ |P{1.5cm} || P{1.25cm} |P{1.25cm} || P{1.25cm} | P{1.25cm}|| P{1.25cm} |P{1.25cm} |}
\hline
\multirow{2}{*}{Tolerance}   &  \multicolumn{2}{c||}{\ttfamily{EI-K4}} &  \multicolumn{2}{c||}{\ttfamily{ode45}} &
\multicolumn{2}{c|}{\ttfamily{Radau}} \\   
\cline{2-7} 
&  $M$ & Time & $M$  & Time & $M$  & Time\\
\hline 
$10^{-2}$ & 320   & 3.56s & 1067777 & 420.77s & 1663 & 369.91s\\ 
$10^{-5}$ & 2560 & 26.67s & 1068325 & 432.29s & 3043 & 699.78s  \\
$10^{-8}$& 20480 & 215.86s & 1073297 & 438.69s & 5234 & 1264.90s\\ 
\hline 
\end{tabular} \caption{Comparison among {\ttfamily{EI-K4}}, {\ttfamily{ode45}}, and {\ttfamily{Radau}} for a semilinear beam equation. The number of time steps for the exponential integrator is empirically chosen to reach the prescribed accuracy.} \label{tab-ode45-semilinear-beam}
\end{table} 

\end{ex}

\begin{ex} \textit{Adding all damping terms into the nonlinear part.} \label{ex-compare-addnonlinear}
Another common approach for solving \eqref{eq-example-wave} consists in merging the damping terms with the nonlinear terms and then employing various exponential integrators to solve the resulting semilinear equation. In this case, we have to solve the system 
\begin{align*} 
\dot{y}(t) = \widetilde{A} y(t) + \widetilde{F}(y(t)),
\end{align*}
where
\begin{align*} 
\widetilde{A} = \bmat{ 0  & I \\ -\alpha S - \delta I & 0 }, \quad \widetilde{F}(y(t)) = \widetilde{F} \pmat{u \\ w} =  \bmat{0 \\ g(u) + h(w) - \beta Sw - \gamma w }. 
\end{align*}
The matrix $S$ is either $S = S_w$ as in \eqref{eq-discrete-Sw} or $S = S_b$ as in \eqref{eq-discrete-Sb}.
The exponential function of $\widetilde{A}$ can be computed by the following formula 
\begin{align} \label{eq-exp-undamped}
\me^{t\widetilde{A}} = \bmat{\cos(t\Omega) & \Omega^{-1} \sin(t\Omega) \\ -\Omega \sin(t\Omega) & \cos(t\Omega)}, \quad \Omega = \sqrt{\alpha S + \delta I}, 
\end{align}
An explicit form of $\varphi_k(t\widetilde{A})$ was also presented in the literature (for example: see \cite{WangXu18}). We illustrate by two examples below the claim that this approach is more expensive and it also lacks accuracy.

 Consider a wave equation \eqref{eq-example-wave} with $\alpha=1,~\beta = 10^{-2},~\delta = 1,~\gamma = 10^{-1}$. The nonlinear term is $g(u) = -5u^3$. The initial conditions are $p(x) = 5\sin(5\pi x)$ and $q(x) = 5\cos(10 \pi x)$. We compute the solution at $T=1$ by using {\ttfamily{EI-SW21}} ($c_2=\frac{1}{3}$) with $M \in \{10,20,\dots, 10 \cdot 2^{10}\}$ time steps. We compare these numerical results with the reference solution evaluated by {\ttfamily{EI-K4}} with $M = 100000$ time steps. The convergence rates of two approaches are plotted in Figure \ref{fig-ex-adddamping-wave}. When we add all damping terms into the nonlinear part, the corresponding approximation is worse than the one given by our approach. Note that the errors are reduced approximately 100 times with our approach. Moreover, with a bigger $\Delta t$, we obtain the solution with an acceptable accuracy while the traditional approach needs much smaller time steps to achieve stability.

We next repeat example \ref{ex-beam1}. The solution at final time step $T=1$ is computed by using {\ttfamily{EI-SW21}} ($c_2=0.2$) with $M \in \{320,\dots, 320 \cdot 2^{7}\}$ time steps. These numerical solutions are compared with the reference solution obtained by using {\ttfamily{EI-SW4}} with $M = 100000$ time steps. The errors are plotted in Figure \ref{fig-ex-adddamping-beam}. While our approach works and preserves the convergence rate of the exponential integrator {\ttfamily{EI-SW21}}, the traditional approach fails even with a small $\Delta t$. Since the matrix $S_b$ is stiffer than the matrix $S_w$, it leads to the stability problem when the structural damping term $\partial_{xxxxt} u$ is added into the nonlinear part.   

In conclusion, the two examples clearly demonstrate the importance of using the matrix exponential of the linearization.
\begin{figure}[h]
\centering 
\begin{subfigure}[]{0.495\textwidth}
         \centering
         \includegraphics[width=\textwidth]{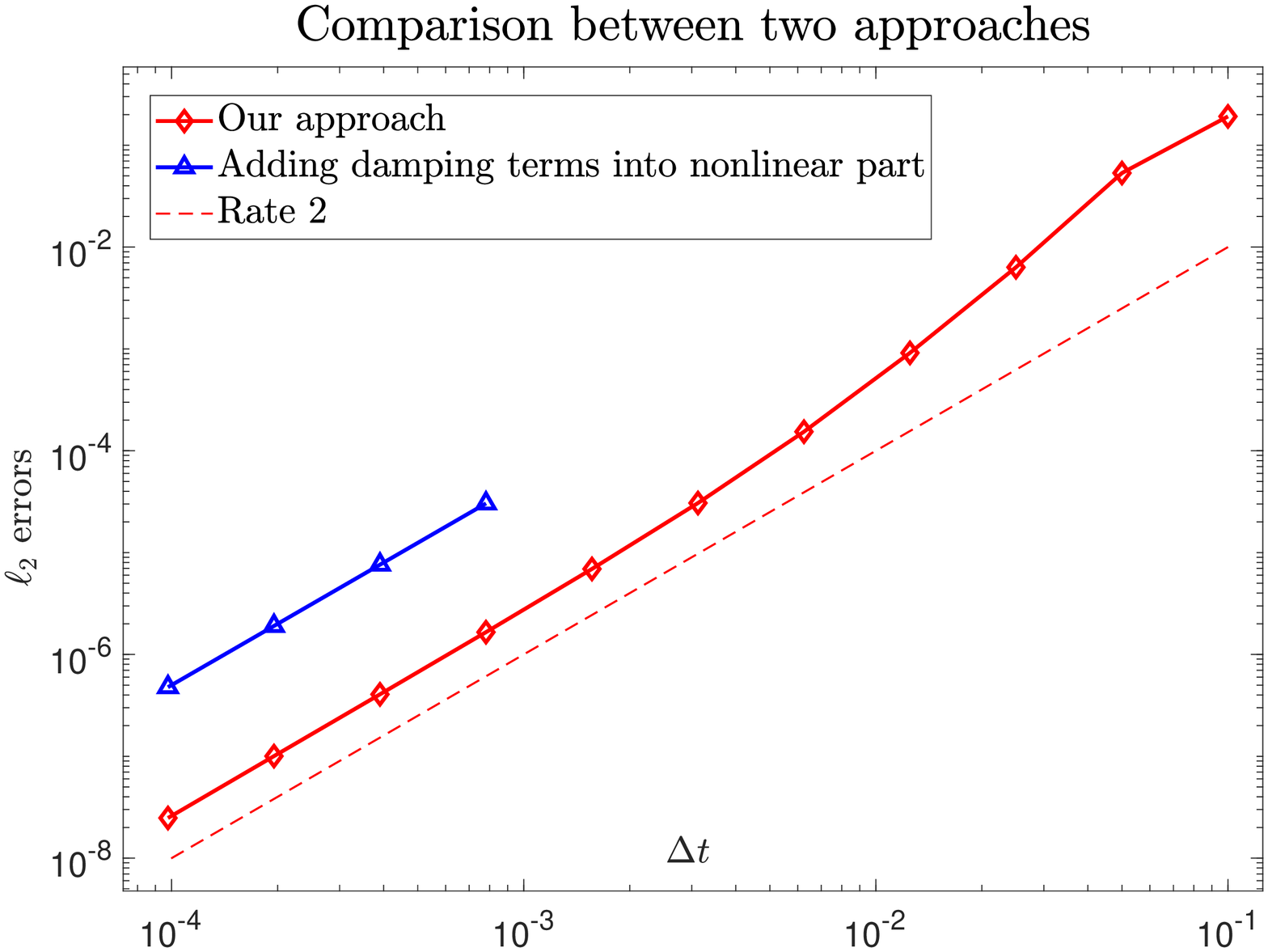}
         \caption{For a wave equation.}
         \label{fig-ex-adddamping-wave}
\end{subfigure}
\begin{subfigure}[]{0.495\textwidth}
         \centering
         \includegraphics[width=\textwidth]{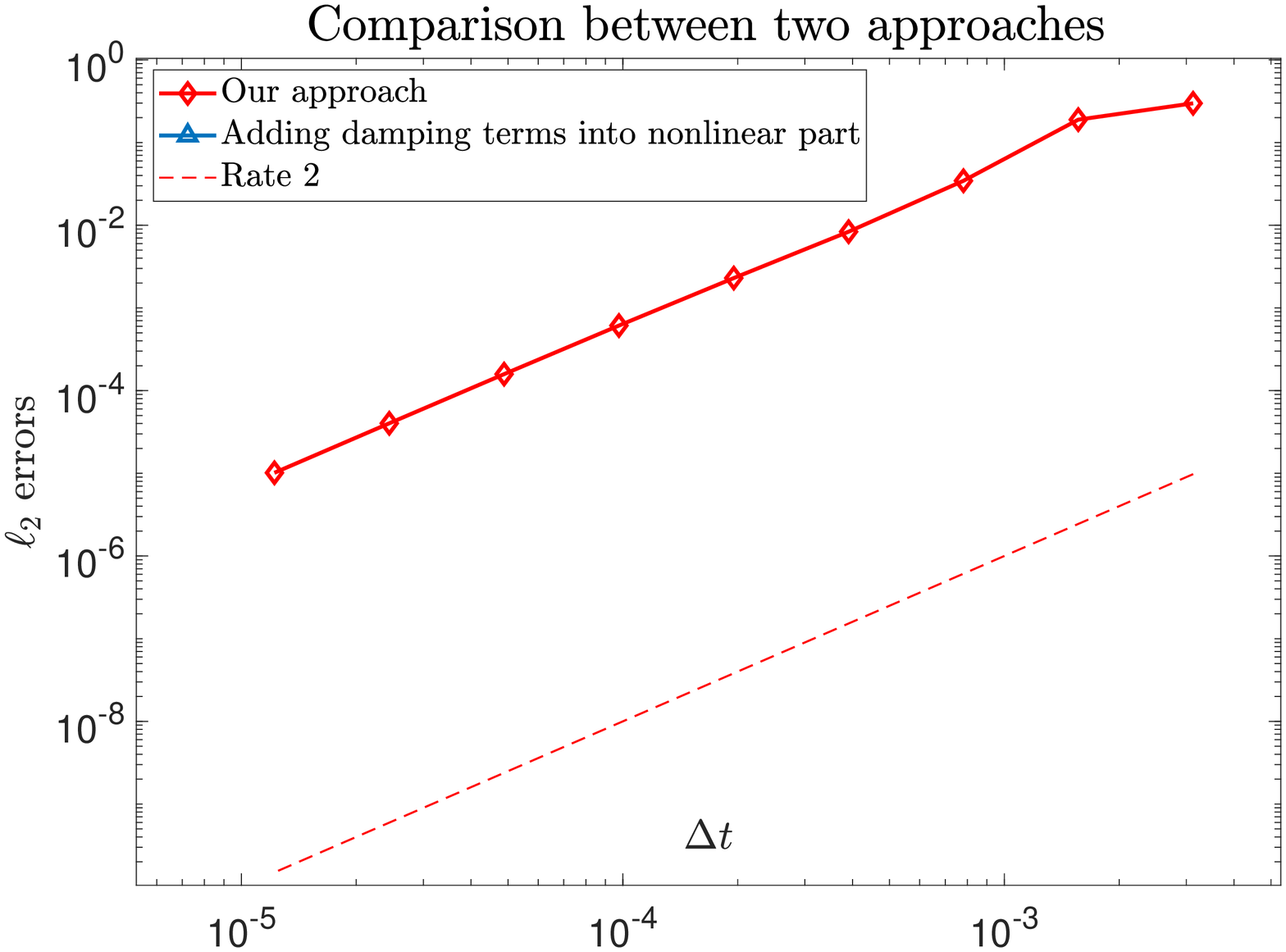}
         \caption{For a beam equation.}
         \label{fig-ex-adddamping-beam}
\end{subfigure}
\caption{The importance of using the matrix exponential of the linearization.}
\end{figure}

\end{ex}

\section{Conclusion}
We presented an approach to  cheaply compute the action of the matrix exponential $\me^{tA}$ as well as the action of the matrix functions $\varphi_k(tA)$ on a given vector by employing two linear transformations. Thus, the solution of certain linear differential equations can be computed in a fast and efficient way. By applying the exponential integrators in the literature, we can solve semilinear wave and semilinear beam equations. 

Note that the described procedure can be extended to the case 
\begin{align*}
A = \begin{bmatrix}
A_1 & A_2 \\ A_3 & A_4
\end{bmatrix}, \quad \text{where~~} [A_i, A_j] = A_i A_j - A_j A_i = 0  \quad \text{for~~} 1 \le i,j \le 4,~i \neq j.
\end{align*}  
Indeed, under the above assumption, the four matrices $A_i$ share the same eigenvalues and eigenvectors. Thus, there exist a matrix $Q$ and four corresponding diagonal matrices $D_i$ such that $A_i = Q D_i Q'$ for $1 \le i \le 4$. This implies that
\begin{align*}
A = \begin{bmatrix}
Q &  0  \\ 0 & Q 
\end{bmatrix}
 \begin{bmatrix}
D_1 &  D_2  \\ D_3 & D_4 
\end{bmatrix} \begin{bmatrix}
Q' &  0  \\ 0 & Q'
\end{bmatrix}
\end{align*}
Thus, the scheme can be analogously applied by evaluating the exponential of each $2 \times 2$ block matrix $\bmat{\lambda_{1,i} & \lambda_{2,i} \\ \lambda_{3,i} & \lambda_{4,i}}$ where $\lambda_{k,i}$ is an entry of the diagonal matrix $D_k$ with $1 \le k \le 4$.

\section*{Acknowledgement} We would like to thank the anonymous referees for the fruitful discussions which lead to improvements in the current version.

\end{document}